\documentclass[aap]{imsart}

\RequirePackage{amsthm,amsmath,amsfonts,amssymb}
\RequirePackage[numbers]{natbib}
\RequirePackage[colorlinks,citecolor=blue,urlcolor=blue]{hyperref}
\RequirePackage{graphicx}

\startlocaldefs
\theoremstyle{plain}

\newtheorem{theorem}{Theorem}[section]
\newtheorem{lemma}[theorem]{Lemma}
\newtheorem{proposition}[theorem]{Proposition}
\newtheorem{corollary}[theorem]{Corollary}
\theoremstyle{remark}
\newtheorem{definition}[theorem]{Definition}

\newtheorem{remark}{Remark}[section]
\renewcommand{\P}{\mathbb{P}}
\newcommand{\E}{\mathbb{E}}
\newcommand{\V}{\text{Var}}
\newcommand{\pf}{\begin{proof}}
\newcommand{\e}{\end{proof}}
\newcommand{\pe}{\hfill$\square$}

\newcommand{\R}{\mathbb{R}}
\newcommand{\KL}{\text{KL}}

\usepackage{bm}

\endlocaldefs

\begin{document}

\begin{frontmatter}
\title{Entropic conditional central limit theorem and Hadamard Compression}
\runtitle{Entropic limit theorem and Hadamard Compression}

\begin{aug}
\author[A]{\fnms{Zhi-Ming}~\snm{Ma}\ead[label=e1]{mazm@amt.ac.cn}},
\author[A]{\fnms{Liu-Quan}~\snm{Yao}\ead[label=e2]{yaoliuquan20@mails.ucas.ac.cn}},
\author[A]{\fnms{Shuai}~\snm{Yuan}\ead[label=e3]{yuanshuai2020@amss.ac.cn}}
\and
\author[B]{\fnms{Hua-Zi}~\snm{Zhang}\ead[label=e4]{zhanghuazi@huawei.com}\orcid{0000-0000-0000-0000}}

\address[A]{University of Chinese Academy of Sciences\printead[presep={,\ }]{e1,e2,e3}}

\address[B]{Huawei Technologies Co., Ltd\printead[presep={,\ }]{e4}}
\end{aug}
as
\begin{abstract}
	We make use of an entropic property to establish a convergence theorem (Main Theorem), which reveals that the conditional entropy measures the asymptotic Gaussianity.  As an application, we establish the {\it entropic conditional central limit theorem} (CCLT), which is stronger than the classical CCLT. As another application, we show that continuous input under  iterated Hadamard transform,  almost every distribution of the output conditional on the values of the previous signals will tend to Gaussian, and the conditional distribution  is in fact  insensitive to the condition. The results enable us to make a theoretic study concerning Hadamard compression, which provides a solid theoretical analysis supporting the simulation results in previous literature. We show also that the conditional Fisher information can be used to measure the asymptotic Gaussianity.

\end{abstract}



\end{frontmatter}

\section{Introduction}\label{section1}

\subsection{Entropic Central Limit Theorem}
Suppose $\xi_1,\xi_2,\cdots,\xi_n$ are independent and identically distributed (\textit{i.i.d.}) random variables with mean $\mu$ and variance $\sigma^2$. Central limit theorem (CLT) states that the normalized partial sum
\begin{equation}\label{partial sum}
W_n = \frac{\sum\limits_{k=1}^n (\xi_k - \mu)}{\sqrt{n\sigma^2}}
\end{equation}
is asymptotically Gaussian distributed in   weak sense. Different from classical proof of CLT based on tools of characteristic functions, the entropic method, which is started by Linnik (1959) \cite{Lin1959}, provides an information-theoretic proof of CLT by studying the behaviour of entropy $h(W_n)$. Recall that for a random variable $X$ with continuous  density $f$,   the \textit{differential entropy} of $X$ is defined as
\begin{equation}\label{differential entropy}
h(X) = \int_\R -f(x)\log f(x) dx,
\end{equation}
where $\log$ denotes the natural logarithm. The \textit{Kullback-Leibler divergence} (KL-divergence) between two probability densities $f_1,f_2$ is given by
\begin{equation}\label{KL divergence}
\KL(f_1||f_2) = \int_\R f_1(x)\log\frac{f_1(x)}{f_2(x)}dx.
\end{equation}
To measure the Gaussianity of $X$, let $D(X) = \KL(f||g)$, where $g$ is the Gaussian density with the same mean and variance as $X$. Barron(1986)\cite{Bar1986} showed that
\begin{equation}\label{Barron_KL}
\lim\limits_{n\rightarrow\infty}D(W_n) = 0.
\end{equation}
Let $f_n$ denote the density of $W_n$, using Pinsker's inequality and \eqref{Barron_KL} we obtain that
\begin{equation}\label{L1_convergence}
\|f_n - g\|_{L^1} := \int_\R |f_n(x) - g(x)|dx\xrightarrow{n\rightarrow\infty} 0,
\end{equation}
which implies CLT directly. Note that $\eqref{L1_convergence}$ provides a strong convergence of the density of $W_n$ in the sense of $L^1$ norm, while only weak convergence is obtained by classical CLT.

The   proof of \eqref{Barron_KL} is based on the  well-known entropy jump inequality: 
\begin{equation}\label{EJI}
h\left(\frac{X+Y}{\sqrt{2}}\right) \geq \frac{h(X)+h(Y)}{2},
\end{equation}
where the equality holds if and only if $X$ and $Y$ are  Gaussian with the same variance. The difference $h\left((X+Y)/\sqrt{2}\right) - (h(X)+h(Y))/2$ is called \textit{entropy jump}. Let $W_n'$ be an independent copy of $W_n$, taking $X = W_n$, $Y = W_n'$ in \eqref{EJI} we obtain
\begin{equation}\label{increase on h_Sn}
h(W_{2n}) = h\left(\frac{W_n + W_n'}{\sqrt{2}}\right) \geq \frac{h(W_n)+h(W_n')}{2} = h(W_n).
\end{equation}
This indicates that $h(W_n)$ will increase to its supremum, and $W_n$ should converge to a "stable distribution" that makes the "=" hold in \eqref{increase on h_Sn}. On the other hand, the equality in \eqref{increase on h_Sn} holds if and only if $W_n$ is Gaussian distributed, and it is also known that Gaussian maximizes the entropy over all random variables with bounded variance. Therefore, it suggests that
\begin{equation}\label{Barron_h}
\lim\limits_{n\rightarrow\infty} h(W_n) = \frac{1}{2}\log2\pi e,
\end{equation}
which is equivalent to \eqref{Barron_KL} by some simple calculation (see \cite{Carlen1991}). This briefly illustrates the main picture of entropic method.

Because of the importance of entropy jump inequality \eqref{increase on h_Sn}, more researchers focused on intensive study on the amount of entropy jump, which leads to entropic CLT as a direct corollary. To describe it we recall an important parameter for random variables called Fisher information, which has close connection with entropy (as stated in \cite{McKean1966},\cite{Carlen1991},\cite{Johnsonbook}, etc.). More precisely,
for a random variable $X$ with absolutely continuous density $f$, its Fisher information is defined to be
\begin{equation}\label{Fisher information}
J(X) = \int_\R \frac{f'(x)^2}{f(x)}dx.
\end{equation}
The tail variance $L_X(R)$ of $X$ is defined to be
\begin{equation}\label{tail variance}
L_X(R) = \E[(X-\E X)^2\bm{1}_{\{|X-\E X|\geq R\}}].
\end{equation}
Using tools from Ornstein-Uhlenbeck semigroup, Carlen (1991)\cite{Carlen1991} showed that for any $\lambda\in(0,1)$ and normalized random variables $X,Y$ (\textit{i.e.}, $X,Y$ have zero mean and unit variance), if their Fisher information and tail variance are uniformly bounded, then there exists a non-negative $\delta$ such that
\begin{equation}\label{Carlen}
h\left(\lambda X + \sqrt{1-\lambda^2}Y\right) \geq \lambda^2 h(X) + (1-\lambda^2) h(Y) + \delta,
\end{equation}
where $\delta$ relies only on the uniform bounds of Fisher information and tail variance,  as well as the value of $\left(\frac{1}{2}\log 2\pi e-h(X) \right)$ and $\left(\frac{1}{2}\log 2\pi e-h(Y) \right)$. In other words, the entropy jump $\delta$ relies only on some bounds in regular conditions and the  entropy gaps of $X$ and $Y$ to Gaussian, but not on the particular distributions of $X$ and $Y$. In particular, $\delta = 0$ if and only if $h(X) = h(Y) = \frac{1}{2}\log 2\pi e$, or equivalently, $X,Y$ are  Gaussian with the same variance. Ball \textit{et al.} (2003)\cite{gap-EJ} proved a variational inequality for Fisher information and used it to deduce a lower bound on the entropy jump $\delta$.  Johnson \textit{et al.} (2004)\cite{FI-CLT} used the theory of projection in $L^1$ space to obtain an estimate on the decrease of Fisher information, which immediately leads to a lower bound on the entropy jump using de Bruijn identity. Furthermore, the convergence rate of $h(W_n)$ was also estimated in \cite{FI-CLT}.

Entropic CLT has drawn more and more attention in the past decade. Bobkov \textit{et al.} (2014)\cite{Esseenbound} proved a Berry-Esseen type bound for entropic  CLT.   Bobkov \textit{et al.} (2014)\cite{F-CLT} showed a more accurate convergence rate on the standard Fisher information by the Chebyshev-Hermite polynomials. More recently, Li \textit{et al.} (2019)\cite{Renyi} discussed the CLT in the sense of R$\acute{\mbox{e}}$nyi entropy. Johnson (2020)\cite{Johnson2020} relaxed the finite Poincar$\acute{\mbox{e}}$ constant condition for entropy gap theory. Gavalakis (2021)\cite{discrete2021} extended the entropic method to the CLT for lattice-valued discrete random variables.

\subsection{Conditional Central Limit Theorem}\label{polar intruduction}

As emphasized in \cite{yuan2014}, the random nature of many problems arising in the applied sciences leads to mathematical models where conditioning is
present, thus it is necessary to study the properties of the conditional distribution.

Holst (1979) proved a CCLT for a special conditional distribution in \cite{Holst1979TwoCL}, which has a simple form that under some regular assumptions for $(\xi_i, \eta_i)$,
$$\P\left( \sum_{i=1}^n \frac{\xi_i}{\sqrt{\sum_{i=1}^n\V(\xi_i)}}\le t|\sum_{i=1}^n \eta_i=n\E \eta_1+o(1)\right) \to \Phi(t), \;\; n\to\infty,$$
where $\eta_i$'s are \textit{i.i.d.}and discrete, $\xi_i$'s are also \textit{i.i.d.}with $E \xi_i=0,\forall i$, $\Phi$ is standard Gaussian distribution function. The result was strengthened in \cite{Janson2001MomentCI} to $n\E  \eta_1+O(\sqrt{n}).$

In \cite{Rubshtein1996} Rubshtein proved a CCLT for ergodic stationary process. Using martingale method, Rubshtein showed that for a two-dimensional ergodic stationary process $(\bm{\xi},\bm{\eta}):=\{(\xi_n,\eta_n),n\geq  1\}$ satisfying some regular conditions, the conditional distribution function
\begin{equation}
F_{n,\bm{z}}(t) = \P\left(\frac{S_n - \E (S_n|\bm{\eta}=\bm{z})}{\sqrt{\V (S_n|\bm{\eta}=\bm{z})}}\leq t| \bm{\eta}=\bm{z}\right), t\in\R
\end{equation}
converges weakly to the distribution function of standard Gaussian for realizations $\bm{z} =\bm{\eta}(\omega)= \{\eta_n(\omega)\}$ almost surely, where $S_n = \xi_1+\xi_2+\cdots+\xi_n$.

More recently, \cite{yuan2014} showed a limit theorem for the conditionally independent case. \cite{2021Stein} used Stein method to deduce an upper bound for Wasserstein distance between conditional distribution and Gaussian distribution.

\subsection{Hadamard Compression}\label{section1.3}

In 2009, Arikan  found Polar code with successive cancelation (SC) decoding in \cite{Polar2009}, which can achieve the Shannon limitation in GF(2) when block length tends to infinity.
Afterwords,    Polar method was also applied in lossless source compression \cite{Arikansource2010}.
In the discrete inputs case  with polar transform and SC decoding, the conditional entropy of the $k$th signal $Y_k$ under the realization of the previous signals $Y_1=y_1,Y_2=y_2,\cdots, Y_{k-1}=y_{k-1}$ becomes either 0 or 1, which means that some signals can be uniquely determined and thus can be compressed (see \cite{Arikansource2010} for details). While when inputs are continuous random variables with densities, \cite{Renyiinformationdimension} proved that any linear encoding can not achieve lossless compression.

On the other hand,   the design of the encoding matrix in continuous input compression  has been studied for long time.
The designation of encoding matrix can be divided into two approaches.
One approach is random matrix, such as Gaussian matrix and LDPC, that can facilitate theoretical analysis (\cite{GME}, \cite{2006Compressed}  and \cite{2008Compressed} etc.). Another approach is deterministic matrix, such as Hadamard matrix,  unitary geometry method and Reed-Muller based matrix (  \cite{ULE}, \cite{2014Hadamard}, \cite{2011Hadamardpaper},      \cite{ArikanrealF2021}, \cite{UnitaryGeometry}, \cite{RM-compress1} and \cite{RM-compress2} etc.). The first approach has  rich theoretical results but   poor practicality since its encoding is random.
The second approach   shows directly some practical results with applications, but to the best of our knowledge, most of the practical results were obtained by simulation  experiments  and there were lack of theoretical analysis.

To investigate analog compression of continuous random inputs, \cite{ULE}  employed deterministic
Hadamard matrix transform to replace Polar transform, which is  also considered in \cite{2011Hadamardpaper} and \cite{ArikanrealF2021}.

In Chapter 4 of \cite{ULE}, the signal $\bm{\eta}=(\eta_1,\eta_2,\cdots, \eta_{2^n})^T$ to be compressed by Hadamard transform was supposed to have the form below
\begin{equation}\label{model in ULE}
\bm{\eta}= H_n\bm{\tilde{\xi}}+\bm{G}=H_n(\tilde{\xi}_1,\tilde{\xi}_2,\cdots,\tilde{\xi}_{2^n})^T+(G_1,G_2,\cdots,G_{2^n})^T\overset{d}{=}H_n(\bm{\tilde{\xi}}+\bm{G}),
\end{equation}
where $\bm{\tilde{\xi}}$ is the original signal and $\bm{G}$ is an independent Gaussian noise, and both have \textit{i.i.d.} components respectively. \cite{ULE} aimed to use
high-entropy position to recover low-entropy position. From numerical examples the author
deduced that a high entropy subset of $\eta_i$'s are sufficient to recover all $\bm{\eta}$. This deduction was
verified by further examples of sampling and obtained a satisfactory performance.

Before moving ahead, we introduce some notations. Let $(X, Y)$ be a random pair where $X$ is a random variable and $Y$ may be a multiple-dimensional random vector. If for any realization of $Y=y$, the conditional density of $X$ given $Y=y,$ denoted by $f_{X|Y}(x|y),$ exists and is continuous, then the conditional entropy of $X$ given $Y=y$ is defined as
	\begin{equation}
		h(X|Y=y):=-\int f_{X|Y}(x|y)\log f_{X|Y}(x|y) dx.
	\end{equation}
In this paper we use $h(X|Y)$ to denote the corresponding random variable, and the traditional conditional entropy, i.e. the expectation of $h(X|Y)$, is denoted
explicitly as
	\begin{equation*}
		\E h(X|Y)=-\iint f_{X|Y}(x|y)\log f_{X|Y}(x|y) dx dP_Y(y).
	\end{equation*}
Similar notations will also be used for conditional variance and Fisher information.

For independent random variables $X_1,X_2$ with densities $f_1,f_2$ and $\lambda\in(0,1)$, we write $X_1 *_\lambda X_2:= \lambda X_1 + \sqrt{1 - \lambda^2} X_2$, and denote the density of $X_1 *_\lambda X_2$ by $f_1 *_\lambda f_2$.

Below we consider a slightly more general Hadamard transform. Let $\lambda\in(0,1)$,
\begin{equation}
	H_\lambda=\begin{bmatrix}
		\lambda& \sqrt{1-\lambda^2}\\
		\sqrt{1-\lambda^2} & -\lambda
	\end{bmatrix},
\end{equation}
$H_{\lambda, n}=B_N(H_\lambda)^{\otimes n},$ where $N=2^n$, $n$ is the block length, $B_N$ is the bit permutation matrix, $\otimes$ is the Kronnecker product. We simply denote $H_{1/\sqrt{2}}=H, H_{1/\sqrt{2}, n}=H_n.$

Given a two dimensional vector input $\xi=(\xi_1, \xi_2)^T,$ through Hadamard transform we will get a two dimensional vector output $\eta=(\eta_1,\eta_2)^T= H_\lambda \xi.$ We call $\eta_1$ the upper perform or $"0"$ perform, and call $\eta_2$ the lower perform or $"1"$ perform.  For the n-iteration of Hadamard transform
$H_{\lambda, n}=B_N(H_\lambda)^{\otimes n},$ given an $N$-dimensional input  $\xi^n=(\xi^n_1,\xi^n_2,\cdots, \xi^n_N)^T,$ we will receive an $N$-dimensional output  signal $\eta^n=(\eta^n_1,\eta^n_2,\cdots, \eta^n_N)^T=H_{\lambda,n}\xi^n$.
Each $\eta^n_k$ is obtained by successive $n$ "0" or "1" performs, say, $(\omega_1,\omega_2,\cdots\omega_n),$ with $\omega_i \in \{0,1\}.$ In fact, $(\omega_1,\omega_2,\cdots\omega_n)$ records the polarization path of $\eta^n_k.$

We set $\Omega:=\{0,1\}^\infty$. For $n\ge 1$ and  $\omega =(\omega_1,\omega_2,\cdots\omega_n, \cdots ) \in \Omega,$ we define
\begin{equation}\label{define of k(n,omega)}
	k(n,\omega)=[\omega_1\omega_2\cdots\omega_n]_{2\to 10}+1,
\end{equation}
where $[\cdot]_{2\to10}$ indicates the conversion from  binary to decimal.
We denote
\begin{equation}
	\eta^n_k(\omega)=\eta^n_{k(n,\omega)}.
\end{equation}
Due to the action of  the bit permutation matrix  $B_N,$  one can check that for each polarization path $(\omega_1,\omega_2,\cdots\omega_n),~ \eta^n_k(\omega)$ is exactly the $k$th component $\eta^n_k$ of $ \eta^n.$

We now suppose that the input $\xi^n=(\xi^n_1,\xi^n_2,\cdots, \xi^n_N)^T$ satisfies $\{\xi^n_i\}_{i=1}^N\overset{\textit{i.i.d.}}{\sim} \xi$, where $\xi$ has density $p(x)$. Let  $\eta^n=(\eta^n_1,\eta^n_2,\cdots, \eta^n_N)^T=H_{\lambda,n}\xi^n$ be the received signal.
For $i<j,$ we write $\eta_{i:j}^n$ for $ (\eta^n_i, \eta_{i+1}^n,\cdots, \eta^n_j).$ We denote the density of $\eta_k^n$ as $f_k^n$, the joint density of $\eta_{i:j}^n=(\eta^n_i, \eta_{i+1}^n,\cdots, \eta^n_j)$ as $f_{i:j}^n$, and the conditional density of $\eta_{k}^n$ given  $\eta_{1:k-1}^n=(y_1, y_2, \cdots, y_{k-1})$ as $f^n_{k|y}$ with $y$ stands for $(y_1, y_2, \cdots, y_{k-1})$.
Note that $H_\lambda$ is an orthogonal transform.  Repeatedly employing the chain rule of entropies and the fact that conditional entropies are invariant  under orthogonal transformations. We will find that for all $n$ and $\omega,$

 \begin{equation}\label{core for general n}
 	\begin{aligned}
 		&|\E h(\eta^{n+1}_{k(n+1,\omega)}|\eta^{n+1}_{1:k(n+1,\omega)-1} )-\E h(\eta^{n}_{k(n,\omega)}|\eta^{n}_{1:k(n,\omega)-1} )|\\
 		&=\E h(\eta^{n}_{k(n,\omega)}\ast_\lambda (\eta^n_{k(n,\omega)})'|\eta^{n}_{1:k(n,\omega)-1},(\eta^{n}_{1:k(n,\omega)-1})' )-\E h(\eta^{n}_{k(n,\omega)}|\eta^{n}_{1:k(n,\omega)-1} ),
 	\end{aligned}
 \end{equation}
 where $X'$ denotes an \textit{i.i.d.} copy of  $X$.

 If we fix $\omega$ and  denote
 \begin{equation}\label{Xn, Yn}
 	X_n=\eta^n_{k(n,\omega)},\;\; Y_n=(\eta^n_{1},\eta^n_{2},\cdots,\eta^n_{k(n,\omega)-1})=\eta^n_{1:k(n,\omega)-1},
 \end{equation}
we may rewrite  \eqref{core for general n} as

\begin{align}\label{basic1}
	|\E_{Y_{n+1}}h(X_{n+1}|&Y_{n+1})-\E_{Y_n} h(X_n|Y_n)| \\
= &\E_{Y_n,Y'_n}h(X_n *_\lambda X'_n|Y_n,Y'_n) -  \E_{Y_n} h(X_n|Y_n). \nonumber
\end{align}
The detail proof of \eqref{basic1} will be given in Appendix \ref{A.1}. Note that \eqref{basic1} holds for any polarization path $\omega$, thus it reflects an universal
 property of signal random variables in Hadamard compression.


\section{Paper Structure and Main Results}\label{section2}
 The remainder of this paper is organized as follows.

 In Section \ref{section3},  we provide a Main Theorem  (Theorem \ref{main theorem}) of this paper. We show that if a sequence of random pairs $\{(X_n, Y_n)\}_{n=1}^\infty$  satisfies  \eqref{basic1}, then with some mild additional conditions,  the average of the conditional entropy $h_n := \E h(X_n|Y_n)$ will converge to the entropy of Gaussian with the variance $\sigma^2 :=\lim_{n\to\infty}\E\V(X_n|Y_n)$ (see Theorem \ref{main theorem} for details). Thus,  the conditional entropy measures the asymptotic Gaussianity. This theorem reveals a phenomena that for a sequence of random pairs $\{(X_n, Y_n)\}_{n=1}^\infty$  satisfying the entropic property \eqref{basic1}, one may deduce some results similar to CCLT.   We emphasize that the phenomena revealed by this Main Theorem is valid not only on Hadamard transform, but also on any sequence which satisfies the entropic property \eqref{basic1}.


 The proof of the Main Theorem relies heavily on a Key Lemma (Lemma \ref{extended Carlen}), in which we extend  Theorem 1.2 in \cite{Carlen1991} (see \eqref{Carlen}) to random variables with distinct variances.  We show that for $X, Y$ with uniformly lower bounded variance,  the entropy jump $\delta$  depends only on  the difference between their variance and their gaps to Gaussian. Note that without equal variance condition, the proof of Lemma \ref{extended Carlen} is quite different from that of  \cite{Carlen1991}.  We treat the variance, as well as the differential entropy and the KL divergence  as  functionals in density space and find its continuity with respect to the density. See Proposition \ref{continuity on h V KL} for details.

Section \ref{section4} is devoted to establish the {\it entropic conditional central limit theorem}, which is an important application of the Main Theorem. Let $ (X_n,Y_n)$ be defined by \eqref{Xn, Yn}, if we specify the path $\omega=(0,0,\cdots,0),$  then the conditional distribution of $X_n$ given $Y_n=y,$ denoted by $(X_n|Y_n=y),$  will take the form
 \begin{equation}\label{main to sum1}
 (X_n|Y_n=y)=(\frac{\xi_0^1+\xi_0^2+\cdots+\xi_0^{2^n}}{2^{n/2}}|\eta_0^1=y_1,\eta_0^2=y_2,\cdots ,\eta_0^{2^n}=y_{2^n}),
 \end{equation}
 where $\{(\xi_0^i, \eta_0^i), i=1,2,\cdots,2^n\}$ are \textit{i.i.d.}, $(\xi_0^i, \eta_0^i)\overset{d}{=}(\xi_0, \eta)$, $y=(y_1,y_2,\cdots,y_{2^n})$. Note that $\eta_0^i$ can be considered as a 'side information' of input signal $\xi_i$.
 Therefore, we can make use of Theorem \ref{main theorem} to establish the CCLT.

 To be more precise, suppose that $(\xi,\eta)$ is a two-dimensional random variable such that for almost sure  realization $\eta=y$, the regular conditional distribution $\xi|\eta=y$ exists with finite variance and is absolutely continuous with respect to Lebesgue measure. Let $(\xi_1,\eta_1),(\xi_2,\eta_2),\cdots,(\xi_n,\eta_n)$ be \textit{i.i.d.} samples from $(\xi,\eta)$, $\bm{\eta}_n = (\eta_1,\eta_2,\cdots,\eta_n)$ be the random vector of the first $n$ samples of $\eta$, and $\bm{y}_n = (y_1,y_2,\cdots,y_n)$ be a realization of $\bm{\eta}_n$. Let $W_n = (\xi_1+\xi_2+\cdots+\xi_n)/\sqrt{n}$ be the normalized partial sum, and $\tilde{g}$ be a Gaussian density with the same mean and variance as the conditional distribution $(W_n|\bm{\eta}_n=\bm{y}_n)$. Denote the conditional density of $(W_n|\bm{\eta}_n=\bm{y}_n)$ by $f_n(x|\bm{y}_n)$. Let
 \begin{equation}\label{nconditional Gaussianity1}
 	D(W_n|\bm{y}_n) := \KL(f_n(x|\bm{y}_n)||\tilde{g}) = \int_\R f_n(x|\bm{y}_n)\log\frac{f_n(x|\bm{y}_n)}{\tilde{g}(x)}dx.
 \end{equation}
 Equivalently, $D(W_n|\bm{y}_n)$ is also given by
 \begin{equation}
 	D(W_n|\bm{y}_n) = \KL(\tilde{W}_n(\bm{y}_n)||\mathcal{N}(0,1)|\bm{\eta}_n = \bm{y}_n),
 \end{equation}
 where
 \begin{equation}\label{Sn'1}
 	\tilde{W}_n(\bm{y}_n) = \frac{\sum\limits_{k=1}^n (\xi_k - \E(\xi|\eta_k=y_k))}{\sqrt{\sum\limits_{k=1}^n \V(\xi|\eta_k = y_k)}}
 \end{equation}
 stands for the conditional normalized partial sum. Then, with some mild conditions,  in Section \ref{section4} we will prove the following result (see  Theorem \ref{sum theorem} and Corollary \ref{sum corollary} for details).
 	\begin{equation}\label{nmain result_a1}
 	\lim\limits_{n\rightarrow\infty}\E D(W_n|\bm{\eta}_n) = 0.
 \end{equation}

Note that the entropic CLT \eqref{Barron_KL} is covered by  \eqref{nmain result_a1} if we take $\eta$ to be independent of $\xi$. Since the Gaussianity of $W_n$ conditioning on $\bm{\eta}_n = \bm{y}_n$ is measured by  \eqref{nconditional Gaussianity1},  hence  \eqref{nmain result_a1} provides also the asymptotic Gaussianity of $(W_n|\bm{y}_n)$ which can not be easily deduced by classical CLT.  Moreover, our result  is in some sense stronger than the corresponding result in \cite{Rubshtein1996}
 where only the weak convergence of conditional distribution $\tilde{W}_n(\bm{y}_n)$ is obtained. However,  the convergence in KL-divergence given by  \eqref{nmain result_a1} leads to $L^1$-convergence of the density of $(W_n|\bm{\eta}_n)$ by Pinsker's inequality (cf.$\eqref{L1_convergence}$ ).

   In Section \ref{section4} we  prove also a convergence rate for the entropic conditional CLT, which is stimulated by the rate theorem for entropic CLT in \cite{Esseenbound}. See Theorem \ref{rate theorem} for details.

 In Section \ref{section5}, we study the conditional Fisher information.
 Recall that for a random variable X with differentiable density, the Cramer-Rao bound (CR bound)
 reads
 \begin{equation}\label{CR bound}
 J(X)\geq \frac{1}{\V(X)},
 \end{equation}
 where the equality holds if and only if X is Gaussian (see e.g. Lemma 1.19 in \cite{Johnsonbook}).
 Therefore apart from differential entropy, the Fisher information is another useful quantity to measure Gaussianity.
 It is
 natural to ask whether the conditional Fisher information can also be used to measure the asymptotic
 Gaussianity. To answer this question we need some kind of continuity of Fisher information against
 the densities.
 To  our knowledge  the continuity of Fisher information has not been investigated yet. Although the lower semi-continuity under weak topology was shown in \cite{F-CLT}, it is not applicable to
 our purpose, since we need the continuity of Fisher information against densities equipped with $L^1$-topology. The difficulty lies here is that the definition of Fisher information involves derivatives
 and it is notably known that $L^1$ norm does not coincide with $H^1$ norm. Thus, we consider only
 the densities of random variables with small Gaussian perturbations (cf. \eqref{IGD-n}).  We show that for such random variables  the continuity of Fisher information holds (see Lemma \ref{continuity of Fisher information}), and then the expectation of the conditional Fisher information $\E J(X_n|Y_n)$
 converges to    $1 / \sigma^2$ where $\sigma^2= \lim_{n\to\infty}\E \V(X_n|Y_n)$.  See  Theorem \ref{Fisher limit} for details.

 In Section \ref{section6}, we investigate  Hadamard compression for continuous input.  We show that if input signals are the same as  in  \cite{ULE}, then under Hadamard transform and SC decoding frame, almost every distribution of the received signal conditional on the knowledge of the previous signals   tends to Gaussian, when the  block length tends to infinity, see Theorem  \ref{polar theorem} for details. Moreover, the conditional distribution  is insensitive to the values of previous received signals (cf. \eqref{polarVn}). We studied also the limit relation between variance and Fisher information (Theorem \ref{polar theorem1}).  In Theorem  \ref{V est} we make a rough estimate of the limit variance, which has close link to the error analysis in compression.

  As we have mentioned in Section \ref{section1.3},  from numerical examples the author in \cite{ULE}
 deduced that a high entropy subset of received signal $\eta_i$'s are sufficient to recover all received signals. The deduction was
 verified by further examples of sampling and obtained a satisfactory performance. However, there was no theoretical deduction supporting this performance. It is not clear why those  subset of $\eta_i$'s with lower entropy can be compressed.
 Note that for continuous signals, low entropy does not necessarily imply small variance, and hence does not necessarily mean small uncertainty (see the example in Appendix \ref{a ex}).   Indeed, in the literature (see e.g. \cite{ULE},\cite{2011Hadamardpaper}, \cite{ArikanrealF2021} and references therein ), there was no theoretic study concerning Hadamard compression.

   Based on the research in this section, we can  provide a theoretic study concerning Hadamard compression.   Firstly, by \eqref{polarhn} or \eqref{polarhn1}, for almost every polarization path  $\omega,$ when $n$ tends to $\infty,$ the distribution of received signal $\eta^n_k(\omega)$  conditional on the values of  previous received signals  converges to a Gaussian distribution with corresponding conditional variance. Secondly, by \eqref{polarVn} we know that  the corresponding conditional variance $\V(\eta^n_k(\omega)|\eta^n_{1:k(\omega)-1})$ converges to a deterministic value $V_\infty(\omega).$  Third, the conditional distribution  $(\eta^n_k(\omega)|\eta^n_{1:k(\omega)-1})$ is insensitive to the values of previous received signals in the sense that,  for almost every realization $\eta^n_1=y_1,\eta^n_2=y_2,\cdots,\eta^n_{k-1}=y_{k-1},$ the conditional distribution of $(\eta^n_k(\omega)|\eta^n_1=y_1,\eta^n_2=y_2,\cdots,\eta^n_{k-1}=y_{k-1})$ converges  to  the Gaussian distribution with same variance  $V_\infty(\omega).$ This  can be seen from the deduction below. For almost every $\omega\in \Omega$, when $n$ is sufficiently large, there exists $\sigma^n_k(\omega)>0$, where $k=k(n,\omega)$, such that for almost every realization $(y_1,y_2,\cdots, y_{k-1})$,
 \begin{align*}
 	&\E h(\eta^n_k(\omega)|\eta^n_1,\eta^n_2,\cdots,\eta^n_{k-1})\overset{(a)}{\approx }\frac{1}{2}\log(2\pi e (\sigma^n_k(\omega))^2)\\
 	&\overset{(b)}{\approx }\frac{1}{2}\log(2\pi e \V(\eta^n_k|\eta^n_1=y_1,\eta^n_2=y_2,\cdots,\eta^n_{k-1}=y_{k-1})).
 \end{align*}
 where $(a)$ comes from  \eqref{polarhn}, and $(b)$ comes from  \eqref{polarVn}.

Now, we can conduct a solid theoretical analysis supporting the simulation results of \cite{ULE}.  When $n$ is large enough,  for any $k,~1\le k \le 2^n,$ the received signal $\eta^n_k$ is roughly Gaussian distributed with variance, say $\sigma^n_k,$ and $\sigma^n_k$ is irrelevant to the previous received signals. It is well known that the entropy $h(\eta)$ of a Gaussian variable $\eta$ is uniquely determined by its variance $\sigma$, that is, $h(\eta)= \frac{1}{2}\log(2\pi e \sigma^2).$  Consequently, those $\eta^n_k$'s with lower entropy have less variance, and hence have smaller uncertainty. Therefore one can safely use  high entropy subset of $\eta^n_k$'s  to recover all the received signals, as performed in \cite{ULE}.  Further more, since the conditional distribution is insensitive to the realization of previous signals,  when $n$ is large enough, the choice of compressed positions can be fixed for whatever input signals. This means  that the experimental Hadamard  compression performed  in \cite{ULE} exhibits universal applicability across diverse input signals.

\section{Main Theorem and Key Lemma}\label{section3}

\subsection{Statement of Main Theorem}
We first recall some notations and conventions used in this paper.  Throughout this paper, ''$\log$'' denotes the natural logarithm. $g_a$ stands for the Gaussian density with mean 0 and variance $a$, and $g$ for standard Gaussian density. The letter $G$ will always stands for a standard Gaussian random variable. We denote by $\mathcal{T}$ the collection of all absolutely continuous probability densities, \textit{i.e.},
\begin{equation}\label{T}
\mathcal{T} = \{f:\ f \text{ is absolutely continuous and }f(x)\geq0,\ \int_\R f(x)dx = 1\}.
\end{equation}
We write $L^1_{f(x)}$ for $L^1(\R, f(x)dx)$, in
particular, we write $L^1$ for $L^1(\R,dx)$.

For a random pair $(X,Y)$ with $X$ being a random variable, we always assume that the conditional density $f_{X|Y}(\cdot|y)$ exists and  $f_{X|Y}(\cdot|y)\in\mathcal{T}$ for  almost surely $y$. We write
\begin{equation}
\V(X|Y)=V(Y):=\E(X^2|Y)-[\E(X|Y)]^2
\end{equation}
for the random variable representing the conditional variance under random realization, and we write
$\V(X|y)=V(y)$ for the conditional variance under given realization $y$.
In addition, the average of conditional variance is denoted by $\E\V(X|Y)$. We make the same convention for all functionals operating on  conditional distributions. For example, $h(X|Y)$ is a random variable standing for the conditional entropy under random realization, and the average of $h(X|Y)$ is given by
\begin{equation}\label{conditional entropy}
\E h(X|Y) = \int \int_\R -f_{X|Y}(x|y)\log f_{X|Y}(x|y)dx P_Y(dy).
\end{equation}
Similarly, The average of conditional Fisher information $J(X|Y)$ is denoted by
\begin{equation}\label{conditional Fisher information}
\E J(X|Y) = \int \int_\R \frac{f_{X|Y}'(x|y)^2}{f_{X|Y}(x|y)}dx P_Y(dy).
\end{equation}

Throughout this paper, $l$ stands for a decreasing function $l(R)$ such that $l(R) \rightarrow0$ as $R\rightarrow\infty$. For $\lambda\in(0,1)$ and independent random variables $X,Y$, their entropy jump is defined by
\begin{equation}
\triangle_\lambda(X,Y) := h\left(\lambda X + \sqrt{1-\lambda^2}Y\right) - (\lambda^2h(X)+(1-\lambda^2)h(Y)).
\end{equation}

For a random pair $(X,Y)$, the Gaussianity of $X$ conditioning on $Y=y$ is measured by
\begin{equation}
D(X|Y=y):=\KL[(X|Y=y)\| \mathcal{N}(\mu_y,\sigma^2_y)]=\int f(x|y)\log \frac{f(x|y)}{g(x)}dx,
\end{equation}
where $\mu_y=\E(X|Y=y), \sigma^2_y=\V(X|Y=y), f(x|y)$ is the density of $X$ given $Y=y$, and $g$ is the density of $\mathcal{N}(\mu_y,\sigma^2_y)$.

In the Main Theorem stated below, instead of using traditional tools from probability theory such as characteristic function or martingale, we make use of the entropic property \eqref{basic1} to establish a result similar to conditional CLT. In fact it is even stronger than traditional CCLT,  because by Pinsker's inequality the convergence  in KL divergence sense  implies the convergence in total variance sense rather than in weak sense (cf. \eqref{L1_convergence})

\begin{theorem}[Main Theorem]\label{main theorem}
	Let $\{(X_n, Y_n)\}_{n=1}^\infty$ be a sequence of random pairs such that the conditional distribution $X_n|Y_n=y_n$ has absolute continuous density for $y_n$ almost surely.  Assume the following conditions hold,
	\begin{enumerate}
		\item There exists $a>0$ such that $\V(X_n|Y_n) \geq a,\ a.s.,\ \forall n\in\mathbb{N}^+$ (\text{here and henceforth, } $\mathbb{N}^+$ stands for positive integers).
		\item There exists $J_0 < \infty$ such that $\sup\limits_n\E J(X_n|Y_n)\leq J_0$.
		\item There exists function $l$ decreasing to 0 such that
		$\E[L_{X_n|Y_n}(R)]\leq l(R),\ \forall n \in\mathbb{N}^+, R\geq0.$
	\end{enumerate}
	Denote by $(X_n', Y_n')$ an independent copy of $(X_n, Y_n)$. If
	\begin{equation}\label{entropy condition 1}
	\lim\limits_{n\rightarrow\infty} \E h(X_n|Y_n)\ \ \text{exists},
	\end{equation}
	and
	\begin{equation}\label{entropy condition 2}
	|\E h(X_{n+1}|Y_{n+1})-\E h(X_n|Y_n)| = \E\left[h(X_n *_\lambda X'_n|Y_n,Y'_n) -  h(X_n|Y_n)\right]
	\end{equation}
	for some $\lambda\in(0,1)$. Then
	\begin{equation}\label{general D}
	D(X_n|Y_n) \xrightarrow{\text{Pr.}} 0,\ \ n\rightarrow\infty,
	\end{equation}
	and
	\begin{equation}\label{general variance}
	\lim\limits_{n\rightarrow\infty}\E|\V(X_n|Y_n) - \E \V(X_n|Y_n)| = 0.
	\end{equation}
	Furthermore, If $\{h(X_n|Y_n),n\geq 1\}$ is uniformly integrable,  then there exists $\sigma^2>0$ such that
	\begin{align}
	&\lim\limits_{n\rightarrow\infty} \E\V(X_n|Y_n) = \sigma^2 \label{V limit},\\
	&\lim\limits_{n\rightarrow\infty} \E h(X_n|Y_n) = \frac{1}{2}\log 2\pi e \sigma^2 \label{h limit}.
	\end{align}
\end{theorem}
\begin{remark}
	Here are examples for the conditions in Theorem \ref{main theorem}. As stated in Section \ref{section1.3}, condition \eqref{entropy condition 2} is held in Hadamard transform. Further, if $\omega=(0,0,\cdots,0)$, the conditions, that the  uniformly integrability of $\{h(X_n|Y_n),n\geq 1\}$ and the convergence property \eqref{entropy condition 1}, are   held immediately, as proved in Section \ref{section4}. While for general $\omega$,
	if the input signal of Hadamard compression $\xi$ satisfies the form
	\begin{equation}\label{Gauss}
	\xi=\hat{\xi}+G,
	\end{equation}
	where $G$ is Gaussian and independent with $\hat{\xi}$, as assumed in \cite{ULE}, then   the  uniformly integrability of $\{h(X_n|Y_n),n\geq 1\}$ and the convergence property \eqref{entropy condition 1}  are both held for almost any $\omega$ (on a specific probability space). See details in Appendix \ref{A.2}.
\end{remark}

\subsection{Key Lemma and its Proof}
As a core step for proving Main Theorem, in this subsection we prepare a Key Lemma (Lemma \ref{extended Carlen}), which is an Extension of Theorem 1.2 of \cite{Carlen1991}. The results in this subsection are similar with those in \cite{Carlen1991}, the main difference is that we treat the case where the random variables have distinct variances. Note that without equal variance condition, the proof of Lemma \ref{extended Carlen} is quite different from that of  \cite{Carlen1991}.  We treat the variance, as well as the differential entropy and the K-L divergence  as  functionals in density space and find its continuity with respect to the density  (see Proposition \ref{continuity on h V KL}) .

\begin{lemma}[Key Lemma]\label{extended Carlen}
	
	Suppose $X$ and $Y$ are independent random variables with absolutely continuous densities. Assume the following conditions satisfy,
	\begin{enumerate}
		\item $\min\{J(X),J(Y)\} \leq J_0$ for some $J_0 > 0$.
		\item $\max\{\V(X),\V(Y)\} \geq a$ for some $a > 0$.
		\item $\min\{L_X(R),L_Y(R)\} \leq l(R)$ for some non-negative decreasing function $l(R)$ such that $l(R) \rightarrow0$ as $R\rightarrow\infty$.
	\end{enumerate}
	If
	\begin{equation}\label{entropy gap and variance difference}
	\max\{D(X),D(Y)\} + |\V(X)-\V(Y)| \geq \epsilon > 0,
	\end{equation}
	then for any $\lambda\in(0,1)$, there exists a positive number  $\delta = \delta(\lambda,J_0,a,l,\epsilon) $ which does not rely on the distributions of $X$ and $Y,$ such that
	\begin{equation}
	h\left(\lambda X + \sqrt{1-\lambda^2}Y\right) \geq \lambda^2 h(X) + (1-\lambda^2) h(Y) + \delta.
	\end{equation}
\end{lemma}

Before the proof of Lemma \ref{extended Carlen},
we first list some useful results that will be exploited in our proof.

The next lemma shows that Fisher information decreases after convolution. We refer to \cite[Lemma 3.1]{Carlen1991} for its proof.
\begin{lemma}\label{Fisher ineq}
	Let $X_1,\cdots, X_n$ be independent random variables with densities $f_1,\cdots, f_n\in\mathcal{T}$. Suppose $0<\lambda_1,\cdots, \lambda_n<1$ and $\sum\limits_{i=1}^n \lambda_i^2 = 1$,  then
	\begin{equation}
	J\left(\sum\limits_{i=1}^n \lambda_iX_i\right)\leq \sum\limits_{i=1}^n\lambda_i^2J(X_i),
	\end{equation}
	and the equality holds if and only if $f_1, f_2,\cdots, f_n$ are Gaussian densities with the same variance.
\end{lemma}

The de  Bruijn's identity has been widely quoted in the literature of information theory (\textit{e.g.} see  \cite[Theorem 17.7.2]{cover2006elements}), which connects Fisher information and differential entropy. We write de Bruijn's identity as the following lemma.
\begin{lemma}\label{Bruijn's identity}
	Let $X$ be a continuous random variable with finite variance, $G$ be an independent standard Gaussian random variable. Then
	\begin{equation}
	\frac{d}{dt}h(X+\sqrt{t}G) = \frac{1}{2}J(X+\sqrt{t}G).
	\end{equation}
\end{lemma}

The variational formula for KL-divergence, which has also been extensively explored (\textit{e.g.} see \cite[Corollary 4.15]{stephane2013concentration}), will be used several times in our proof. We state it as the next lemma.
\begin{lemma}\label{KL variation}
	For any probability densities $f_1$ and $f_2$,
	\begin{equation}
	\KL(f_1||f_2) = \sup\limits_{\phi} \left(\int_\R f_1(x)\phi(x)dx - \log \int_\R e^{\phi(x)} f_2(x)dx\right),
	\end{equation}
	where the supremum runs over all bounded measurable functions $\phi$.
\end{lemma}

Recall that for  $\lambda\in(0,1)$ and independent random variables $X_1,X_2$ with densities $f_1,f_2$, we write $X_1 *_\lambda X_2 = \lambda X_1 + \sqrt{1 - \lambda^2} X_2$, and denote the density of $X_1 *_\lambda X_2$ by $f_1 *_\lambda f_2$. Recall that $g_a$ stands for the Gaussian density with mean 0 and variance $a.$ Barrow from the idea of \cite{Carlen1991}, we will frequently use the tools of Ornstein-Unlenbeck Semigroup, which is defined as follows.
\begin{definition}\label{OU semigroup}
	Set $t\geq 0$,  define the map $p_t^*: L^1 \rightarrow L^1$:
	\begin{equation}
	(p_t^*f)(x) = \int_\R f(y)g_{(1-e^{-2t})}(x-e^{-t}y)dy.
	\end{equation}
\end{definition}
$\{p_t^*:t\geq0\}$ is the Lebesgue adjoint to Ornstein-Unlenbeck semigroup. It is well known that $\{p_t^*:t\geq0\}$ is strongly continuous on $L^1$. Let $X$ be a random variable with density $f\in\mathcal{T}$, we define $p^*_t X$ to be the random variable with density $p^*_t f$. It is not hard to verify that
\begin{equation}\label{OU-prop1}
p^*_t X \overset{d}{=} e^{-t}X + \sqrt{1 - e^{-2t}} G = X *_{e^{-t}} G
\end{equation}
and
\begin{equation}\label{OU-prop2}
(p^*_tf)(x) = \E[g_{(1-e^{-2t})}(x-e^{-t}X)] = f *_{e^{-t}} g,
\end{equation}
where $\overset{d}{=}$ means equal in distribution. From \eqref{OU-prop1} it is easy to verify the distributive law for  $p^*_t$  and $*_\lambda$, \textit{i.e.}, for any $f_1,f_2\in\mathcal{T}$,
\begin{equation}\label{commutation}
(p^*_t f_1) *_\lambda (p^*_t f_2) = p_t^*(f_1*_\lambda f_2).
\end{equation}

Finally, we note here that for a random variable $X$ with density $f$, we usually indistinguishably write $X$ and $f$ when handling with some functionals operating on the distribution of $X$. For example, we write $\V(f) = \V (X),\ h(f)=h(X), L_f(R) = L_X(R), \triangle_\lambda(f_X,f_Y) = \triangle_\lambda(X,Y)$ and so on.

We now start to prove our Key Lemma. We begin with  defining some functional spaces. Recall that in this paper $l$ always  stands for a decreasing function $l(R)$ such that $l(R) \rightarrow0$ as $R\rightarrow\infty$.
\begin{definition}
	For fixed $t\geq0$ and function $l$,  define
	\begin{align}
	&\mathcal{H}_{t, l} = \{p_t^*f:\ f\in\mathcal{T}, \E(f) = 0, L_f \leq l,\ J(f)<\infty  \}^{cl},\\
	&\mathcal{B}_{t} = \{p_t^*f:\ f\in\mathcal{T}, \E(f) = 0,\ J(f)<\infty \}^{cl},
	\end{align}
	where $\mathcal{T}$ is defined as \eqref{T}$, ''cl''$ denotes the closure in $L^1_{1+x^2}$. In addition,  we consider $\mathcal{H}_{t, l}$ and $\mathcal{B}_{t}$ as subspaces of $L^1_{1+x^2}$,  equipped with its topology.
\end{definition}
The next two propositions aim to establish some regularity for functions lying in $\mathcal{H}_{t,l}$ and $\mathcal{B}_t$.
\begin{proposition}\label{bound on H and B}
	For $t > 0$ and function $l$,
	\begin{enumerate}
		\item
		\begin{equation}\label{proposition 3.1 1}
		\sup\limits_{x\in\R} q(x) \leq \frac{1}{\sqrt{2\pi(1-e^{-2t})}}\ , \ \forall  q\in\mathcal{B}_t.
		\end{equation}
		\item There exists a constant $B_t>0$ such that
		\begin{equation}\label{proposition 3.1 2}
		\frac{[q'(x)]^2}{q(x)} \leq B_t[p_{2t}^*(e^tX)](x)\ , \ \forall q\in\mathcal{B}_t,
		\end{equation}
		where $X$ denotes the random variable with density $f$,  $q = p_t^*f$.
		\item There exists a constant $A_{t, l}>0$ such that
		\begin{equation}
		q(x) \geq A_{t, l} g_{(1-e^{-t})}(x)\ , \ \forall q\in\mathcal{H}_{t, l}.
		\end{equation}
	\end{enumerate}
\end{proposition}
\begin{proposition}\label{L of HtL}
	Given $l$, there exists a function $l'$ decreasing to 0 such that
	\begin{equation}
	L_q \leq l'\ ,\ \forall q \in \mathcal{H}_{t,l}.
	\end{equation}
\end{proposition}

The proofs of Proposition \ref{bound on H and B} and Proposition \ref{L of HtL} are given in Appendix \ref{proof of bound on H and B} and \ref{proof of L of HtL}, respectively. The two propositions above are exploited to prove the compactness of $\mathcal{H}_{t,l}$ as next lemma.
\begin{lemma}\label{compactness of HtL}
	Given $t>0$ and function $l$, $\mathcal{H}_{t, l}$ is compact in $L^1_{1+x^2}$.
\end{lemma}
\pf For any $\epsilon >0$, we show that $\mathcal{H}_{t, l}$ has a finite $\epsilon-$net in $L^1_{1+x^2}$. By Proposition \ref{L of HtL} we can find $R > 1$ such that
\begin{equation}
\int_{|x|>R} (1+x^2)q(x)dx \leq 2\int_{|x|>R} x^2q(x)dx \leq 2l'(R)\leq  \frac{\epsilon}{4}\ , \ \forall q\in\mathcal{H}_{t, l}.
\end{equation}

Proposition \ref{bound on H and B} implies that the functions in $\mathcal{B}_t$ are uniformly bounded and equicontinuous. Since $\mathcal{H}_{t, l}\subset\mathcal{B}_t$, by Arzel\`{a}-Ascoli theorem we conclude that $\mathcal{H}_{t, l}$ is compact in $C([-R, R])$. As a result, there exists a finite $\epsilon/4(R + \frac{R^3}{3})-$net of $\mathcal{H}_{t, l}$ in $C([-R, R])$, which we denote by $W$.

We claim that $W$ is a finite $\epsilon-$net of $\mathcal{H}_{t, l}$ in $L^1_{1+x^2}$. In fact, For any $\ q\in\mathcal{H}_{t, l}$, there exists $q_\epsilon\in W(\subset \mathcal{H}_{t, l})$ such that
\begin{equation}
\max\limits_{x\in[-R, R]}|q(x) - q_\epsilon(x)| \leq \frac{\epsilon}{4(R+\frac{R^3}{3})}.
\end{equation}
Therefore,
\begin{equation}
\begin{aligned}
&\|q - q_\epsilon\|_{L^1_{1+x^2}} = \int_{|x| > R}(1+x^2)|q(x) - q_\epsilon(x)|dx + \int_{|x| \leq R}(1+x^2)|q(x) - q_\epsilon(x)|dx \\
&\leq \int_{|x| > R}(1+x^2)q(x)dx + \int_{|x| > R}(1+x^2)q_\epsilon(x)dx + \frac{\epsilon}{4(R+\frac{R^3}{3})}\int_{|x| \leq R}(1+x^2)dx\\
& \leq \frac{\epsilon}{4}+\frac{\epsilon}{4}+\frac{\epsilon}{2} = \epsilon.
\end{aligned}
\end{equation} \e

Next proposition shows the continuity of differential entropy, variance and KL-divergence from Gaussian.
\begin{proposition}\label{continuity on h V KL}
	Let $t>0$ and function $l$ be given. Restricted on $\mathcal{H}_{t, l}$, the following functionals are continuous.
	\begin{enumerate}
		\item The differential entropy $h(f)$.
		\item The variance $\V(f)$.
		\item The KL-divergence from Gaussian $D(f)$ .
	\end{enumerate}
\end{proposition}

We include the proof of Proposition \ref{continuity on h V KL} in Appendix \ref{proof of continuity on h V KL}.

\begin{lemma}\label{continuity on entropy}
	For any $f\in\mathcal{T}$, the function $t\mapsto h(p_t^* f)$ is continuous on $[0, \infty)$,   differentiable on $(0, \infty)$,  and
	$$\lim\limits_{t\rightarrow\infty} h(p_t^*f) = h(g) = \frac{1}{2}\log(2\pi e)\ \ ,\ \ \frac{d}{dt}h(p_t^* f) = J(p_t^*f) - 1.\ \ $$
\end{lemma}
\pf Notice that $\{p_t^*,t\geq0\}$ is a strongly continuous semigroup,  thus it is enough to show the continuity at $t=0$. Let $X$ be the random variable with density $f$. On the one hand,
\begin{equation}
h(p_t^* f) = h(e^{-t}X+\sqrt{1-e^{-2t}}G) = -t + h(X + \sqrt{e^{2t}-1}G)\geq -t + h(X),
\end{equation}
which implies $\liminf\limits_{t\rightarrow0}h(p_t^* f) \geq h(X) = h(f).$

On the other hand, for any $\epsilon >0$, Lemma \ref{KL variation} implies that there exists a bounded measurable function $\phi(x)$ such that
\begin{equation}
\frac{1}{2}\log(2\pi e\V(f)) - h(f) = D(f) \leq \epsilon + \int_\R f(x)\phi(x)dx - \log\int_\R e^{\phi(x)}g_{\V(f)}(x)dx.
\end{equation}
Let $V_t = e^{-2t}\V(f) + 1-e^{-2t}$, use Lemma \ref{KL variation} again and we obtain
\begin{equation}\label{upper of h}
\begin{aligned}
&\frac{1}{2}\log(2\pi eV_t) - h(p_t^*f) \geq \int_\R \phi(x)(p_t^*f)(x)dx - \log\int_\R e^{\phi(x)}g_{V_t}(x)dx \\
& = \int_R \phi(x)[(p_t^*f)(x)-f(x)]dx + \int_\R f(x)\phi(x)dx- \log\int_\R e^{\phi(x)}g_{V_0}(x)dx\\
&\hspace{3em} + \log\int_\R e^{\phi(x)}g_{V_0}(x)dx-\log\int_\R e^{\phi(x)}g_{V_t}(x)dx\\
&\geq \frac{1}{2}\log(2\pi eV_0) - h(f)-\epsilon - \|\phi\|_\infty\|p_t^*f-f\|_{L^1} \\
&\hspace{3em}+ \log\int_\R e^{\phi(x)}g_{V_0}(x)dx-\log\int_\R e^{\phi(x)}g_{V_t}(x)dx.
\end{aligned}
\end{equation}
Since $\sqrt{V_t}G$ converges weakly to $\sqrt{V_0}G$,
\begin{equation}
\int_\R e^{\phi(x)}g_{V_t}(x)dx = \E e^{\phi(\sqrt{V_t}G)} \xrightarrow{t\rightarrow0}
\E e^{\phi(\sqrt{V_0}G)} = \int_\R e^{\phi(x)}g_{V_0}(x)dx.
\end{equation}
As a result, we have $\limsup\limits_{t\rightarrow0}h(p_t^*f) \leq h(f) + \epsilon$. This completes the proof of the continuity at $t=0$.

For the second part of Lemma \ref{continuity on entropy}, since
\begin{equation}
\begin{aligned}
&h(p_t^*f) = h(e^{-t} + \sqrt{1-e^{-2t}}G) \geq h(\sqrt{1-e^{-2t}}G) = h(g) + \frac{1}{2}\log(1-e^{-2t}),
\\
&h(p_t^*f)\leq \frac{1}{2}\log(2\pi e\V(p_t^*f)) =  \frac{1}{2}\log(2\pi e(e^{-2t}\V(f)+1-e^{-2t})),
\end{aligned}
\end{equation}
we have $\lim\limits_{t\rightarrow\infty} h(p_t^*f) = h(g).$

Finally, according to Lemma \ref{Bruijn's identity},
\begin{equation}
\begin{aligned}
\frac{d}{dt}h(p_t^*f) &= \frac{d}{dt}(-t + h(X + \sqrt{e^{2t}-1}G)) = -1 + \frac{1}{2}J(X + \sqrt{e^{2t}-1}G)2e^{2t} \\
&= -1+ J(e^{-t}X+\sqrt{1-e^{-2t}}G) = J(p_t^*f)-1,
\end{aligned}
\end{equation}
which completes the whole proof.\e

From Lemma \ref{continuity on entropy} we obtain the following useful corollary.
\begin{corollary}\label{integrate Fisher eq entropy}
	Suppose $f\in\mathcal{T}$ and $J(f) < \infty$,  then
	\begin{equation}
	h(p_t^*f) = h(f) + \int_0^t [J(p_s^*f)-1]ds.
	\end{equation}
\end{corollary}

The next lemma shows the operation of $p^*_t$ leads to smaller entropy jump.
\begin{lemma} \label{entropy ineq}
	Suppose $f_1, f_2\in\mathcal{T}$ and $J(f_1)<\infty,J(f_2)<\infty$, $\lambda \in(0, 1)$. Then for any $t>0$,
	\begin{equation}
	\triangle_\lambda(f_1,f_2) \geq \triangle_\lambda(p_t^*f_1,p_t^*f_2),
	\end{equation}
	the equality holds iff $f_1, f_2$ are both Gaussian with the same variance.
\end{lemma}
\pf By Corollary \ref{integrate Fisher eq entropy}, Lemma \ref{Fisher ineq} and \eqref{commutation},
\begin{align*}
& \hspace{1.4em} h((p_t^*f_1)*_\lambda (p_t^* f_2)) = h(p_t^*(f_1*_\lambda f_2)) \\
&= h(f_1*_\lambda f_2) + \int_0^t [J(p_s^*(f_1*_\lambda f_2))-1]ds \\
&= h(f_1*_\lambda f_2) + \int_0^t [J((p_s^*f_1)*_\lambda (p_s^*f_2))-1]ds\\
& \leq h(f_1*_\lambda f_2) + \int_0^t [\lambda^2J(p_s^*f_1) + (1-\lambda^2)J(p_s^*f_2)-1]ds\\
& = h(f_1*_\lambda f_2) + \lambda^2\int_0^t [J(p_s^*f_1)-1]ds+(1-\lambda^2)\int_0^t + J(p_s^*f_2)-1]ds\\
& = h(f_1*_\lambda f_2) + (\lambda^2h(p_t^*f_1) + (1-\lambda^2)h(p_t^*f_2)) - (\lambda^2h(f_1)+(1-\lambda^2)h(f_2)),
\end{align*}
the equality holds iff $p_t^*f_1$ and $p_t^*f_2$ are Gaussian with the same variance for all $t>0$, which is equivalent to the statement of Lemma \ref{entropy ineq}. \e

By taking $t\rightarrow\infty$ in Lemma \ref{entropy ineq} and using Lemma \ref{continuity on entropy}, we obtain the following corollary.
\begin{corollary}\label{entropy power}
	Given $f_1,f_2\in\mathcal{T}$ and $J(f_1)<\infty,J(f_2)<\infty$, $\lambda\in(0,1)$, then
	\begin{equation}
	\triangle_\lambda(f_1,f_2)\geq0,
	\end{equation}
	the equality holds iff $f_1, f_2$ are Gaussian with the same variance.
\end{corollary}

Finally, we can prove Lemma \ref{extended Carlen}.

\textbf{Proof of Lemma \ref{extended Carlen}.} Without loss of generality, we assume $\E(X) = \E(Y) = 0$. Suppose $X$ and $Y$ have density $f_1$ and $f_2$, respectively.  Denote
\begin{equation}
V_i = \V(f_i),\ V_{t,i} = \V(p_t^*f_i) = e^{-2t}\V(f_i) + 1 - e^{-2t},\\ i=1,2.
\end{equation}
Then for $i = 1,2$,
\begin{equation}
\begin{aligned}
&|V_{t,i} -V_i| = |(1-e^{-2t})(1-V_i)| \leq (1-e^{-2t})(V_i+1) \leq (1-e^{-2t})(l(0)+1),\\
&|h(g_{V_{t,i}}) - h(g_{V_i})| =\left|\frac{1}{2}\log(2\pi eV_{t,i}) - \frac{1}{2}\log(2\pi eV_{i})\right| \leq \frac{1}{2a}|V_{t,i}-V_i| \leq \frac{(1-e^{-2t})(l(0)+1)}{2a}.
\end{aligned}
\end{equation}
Therefore, it is able to take  $t_1 = t_1(l,a,\epsilon)$ sufficiently small such that for $i=1,2$,
\begin{equation}
|V_{t,i} -V_i|\leq \frac{\epsilon}{5}\ \ \text{and}\ \
|h(g_{V_{t,i}}) - h(g_{V_i})|\leq \frac{\epsilon}{5},\ \ \forall t < t_1.
\end{equation}
According to Corollary \ref{integrate Fisher eq entropy}, for $i=1,2$ and any $t>0$ we have
\begin{align}
|h(p_t^*f_i) - h(f_i)| \leq& \int_0^t |J(p_s^*f_i) - 1| ds\\ \leq & \int_0^t e^{-2s}J(f_i)+1-e^{-2s}+1\  ds
\leq \int_0^t (J_0-1)e^{-2s}+2 \ ds, \nonumber
\end{align}
Thus we can take $t_2 = t_2(J_0,\epsilon)$ small enough such that for $i=1,2$,
\begin{equation}
|h(p_t^*f_i) - h(f_i)| \leq \frac{\epsilon}{5},\ \ \forall t < t_2.
\end{equation}
Let $t = \min\{t_1,t_2\}$,  then for $i=1,2$,
\begin{equation}
\begin{aligned}
D(p_t^*f_i)  & = h(g_{V_{t,i}}) - h(p_t^*f_i) \\
&=h(g_{V_{t,i}}) - h(g_{V_{i}}) + h(g_{V_{i}}) - h(f_i) + h(f_i) - h(p_t^*f_i)\\
&\geq \KL(f_i||g_{V_i}) - \frac{2\epsilon}{5}= D(f_i) - \frac{2\epsilon}{5},
\end{aligned}
\end{equation}
and
\begin{equation}
|V_{t,1}-V_{t,2}| \geq |V_1-V_2| - |V_{t,1}-V_1|-|V_{t,2}-V_2| \geq |V_1-V_2| - \frac{2\epsilon}{5}.
\end{equation}
Thus
\begin{equation}\label{A}
\begin{aligned}
&\max\{D(p^*_tf_1),D(p^*_tf_2)\}+ |\V(p^*_tf_1)-\V(p^*_tf_2)| \\
&\geq \max\{D(f_1),D(f_2)\}+|V_1-V_2|- \frac{4\epsilon}{5} \geq \frac{\epsilon}{5}.
\end{aligned}
\end{equation}
Denote $\varphi(q_1,q_2) =\max\{D(q_1),D(q_2)\} + |\V(q_1) - \V(q_2)|$, then by Proposition \ref{continuity on h V KL}, $\varphi$ is continuous on $\mathcal{H}_{t,l} \times \mathcal{H}_{t,l}$. Let
\begin{equation}
\mathcal{G} = \{(q_1,q_2)\in\mathcal{H}_{t,l} \times \mathcal{H}_{t,l}:\ \varphi(q_1,q_2) \geq \frac{\epsilon}{5}\},
\end{equation}
from Lemma \ref{compactness of HtL} and the continuity of $\varphi$, we conclude that $\mathcal{G}$ is compact. Notice that the entropy jump $\triangle_\lambda(q_1,q_2)$ is continuous on $\mathcal{H}_{t,l}\times \mathcal{H}_{t,l}$ by Proposition \ref{continuity on h V KL} and the fact that $*_\lambda$ is linear operation. Therefore, it is able to take $(a_1,a_2)\in\mathcal{G}$ such that
\begin{equation}
\triangle_\lambda(a_1,a_2) = \inf\limits_{(q_1,q_2)\in\mathcal{G}}\triangle_\lambda(q_1,q_2).
\end{equation}
By \eqref{A} we have $(p_t^*f_1,p_t^*f_2)\in\mathcal{G}$. Now using Lemma \ref{entropy ineq} we obtain
\begin{equation}
\triangle_\lambda(f_1,f_2)\geq\triangle_\lambda(p_t^*f_1,p_t^*f_2)\geq\triangle_\lambda(a_1,a_2).
\end{equation}
Notice that $\mathcal{G}$ cannot contain a Gaussian pair with the same variance, in other word, it is impossible that  $(a_1,a_2)$ are both Gaussian with the same variance. Thus Corollary \ref{entropy power} implies that $\triangle_\lambda(a_1,a_2) >0$. Now let $\delta = \triangle_\lambda(a_1,a_2)$, it is clear that $\delta$  depends only on $\lambda,J_0,l,a,\epsilon$, but not on the distributions of $X,Y$. This completes the proof.\pe

\subsection{Proof of the Main Theorem \ref{main theorem}}
In this subsection we use Lemma \ref{extended Carlen} to prove Theorem \ref{main theorem}. First, we show that the regular conditions required in Lemma \ref{extended Carlen} holds with high probability.
\begin{proposition}\label{F and L condition}
	Under the assumption 2 and assumption 3 in Theorem \ref{main theorem}, for fixed $ \epsilon > 0$, there exists  constant $J_1 < \infty$ and function $l_1(R)$ decreasing to 0 as $R\rightarrow \infty$, such that
	\begin{equation}
	\P\left(J(X_n|Y_n)\leq J_1,\ L_{X_n|Y_n}(R)\leq l_1(R),\forall R>0\right)   \geq 1 - \epsilon, \ \ \forall n\in\mathbb{N}^+,
	\end{equation}
\end{proposition}

The proof of Proposition \ref{F and L condition} is given in Appendix \ref{proof of F and L condition}.

\textbf{Proof of Theorem \ref{main theorem}.} For any $\alpha >0$, by Proposition \ref{F and L condition} and the assumption 1 in Theorem \ref{main theorem}, there exists $J_1<\infty$ and function $l_1(R)$ such that the set
\begin{equation}
T_n = \{y_n:\ J(X_n|y_n) \leq J_1,\ L_{X_n|y_n}(R)\leq l_1(R),\forall R>0,\ \V(X_n|y_n) \geq a\}
\end{equation}
has probability $\P(T_n) \geq 1 - \alpha$.

We first prove \eqref{general D}. For any $\epsilon > 0,\ n\in\mathbb{N}^+$. Define
\begin{equation}
R_{n} = \{y_n:\ D(X_n|y_n) > \epsilon\}.
\end{equation}
Let $W_n = T_n \bigcap R_n$, by Lemma \ref{extended Carlen}, there exists a constant $\delta>0$ only depending on $J_1,\ l_1,\ a,\ \lambda,\ \epsilon$ such that
\begin{equation}\label{lower bound delta}
\begin{aligned}
\triangle_\lambda(X_n,X'_n|y_n,y'_n)&:=h(X_n *_\lambda X'_n|y_n,y'_n) - (\lambda^2h(X_n|y_n) + (1-\lambda^2)h(X'_n|y'_n)) \\
&\geq \delta,\ \ \forall y_n,y'_n \in W_n.
\end{aligned}
\end{equation}
From \eqref{entropy condition 2} and \eqref{lower bound delta} we have
\begin{equation}\label{H1}
\begin{aligned}
&|\E h(X_{n+1}|Y_{n+1}) - \E h(X_n|Y_n)| =\E\left[h(X_n*_\lambda X_n'|Y_n, Y_n')-h(X_n|Y_n)\right]\\
&  = \int \left[h(X_n *_\lambda X'_n|y_n,y'_n)-(\lambda^2h(X_n|y_n) + (1-\lambda^2)h(X'_n|y'_n))\right]\P(dy_n)\P(dy'_n) \\
& = \int \triangle_\lambda(X_n,X'_n|y_n,y'_n)\P(dy_n)\P(dy'_n)\\
&\geq \int_{W_n \times W_n} \triangle_\lambda(X_n,X'_n|y_n,y'_n)\P(dy_n)\P(dy'_n)\\
& \geq \delta \left(\P(W_n)\right)^2.
\end{aligned}
\end{equation}
Notice that $\delta$ is independent with $n$, and $h(X_n|Y_n)$ is convergent, letting $n\rightarrow\infty$ in \eqref{H1} we have
\begin{equation}
\lim\limits_{n\rightarrow\infty}\  \P(W_n) = 0.
\end{equation}
Since $\alpha$ is arbitrary and
\begin{equation}\label{H2}
\P(W_n) \geq \P(T_n)+\P(R_n) - 1 \geq \P(R_n) - \alpha,
\end{equation}
we obtain $\lim\limits_{n\rightarrow\infty}\P(R_n) = 0$ and this proves \eqref{general D}.

Next we show \eqref{general variance}. For any $\epsilon >0$, define
\begin{equation}
A_{n} = \{(y_n,y'_n):\ |\V(X_n|y_n)-\V(X'_n|y'_n)| \geq \epsilon\}.
\end{equation}
Let $U_n = (T_n\times T_n)\cap A_n$, Lemma \ref{extended Carlen} implies that
\begin{equation}
\triangle_\lambda(X_n,X'_n|y_n,y'_n) \geq \delta, \ \forall (y_n,y'_n)\in U_n
\end{equation}
for some $\delta >0$ that is independent with $n$. Now follow the same argument as \eqref{H1}-\eqref{H2} we obtain
\begin{equation}\label{An vanish}
\lim\limits_{n\rightarrow\infty}\P(A_n) = 0.
\end{equation}
Let $Z_n = \V(X_n|Y_n)$, \eqref{An vanish} is equivalent to
\begin{equation}
Z_n - Z'_n \xrightarrow{\text{Pr.}} 0 ,\ n\rightarrow\infty,
\end{equation}
where $Z'_n$ is the independent copy of $Z_n$. Since $\{Z_n,n\geq 1\}$ is uniformly integrable, we have $Z_n - Z_n' \xrightarrow{L^1} 0$. Let $\mathcal{F} = \sigma(\{Z_n,n\geq 1\})$ be the sigma algebra generated by all $Z_n$, then we have
\begin{equation}
Z_n - \E Z_n = \E\left(Z_n-Z'_n|\mathcal{F}\right) \xrightarrow{L^1} 0,
\end{equation}
which implies
\begin{equation}\label{Zn L1}
\E|\V(X_n|Y_n) - \E\V(X_n|Y_n)| = \E|Z_n - \E Z_n|\rightarrow 0,\ \ n\rightarrow\infty.
\end{equation}

Finally we show \eqref{V limit} and \eqref{h limit}. Denote $h_n = \E h(X_n|Y_n),V_n = \E\V(X_n|Y_n)$.
Let $V_{n_k}$ be any convergent subsequence of $V_n$ and denote the limit by $v$. From \eqref{Zn L1} we have $Z_{n_k} \xrightarrow{\text{Pr.}}v$. According to \eqref{general D} we have
\begin{equation}
D(X_{n_k}|Y_{n_k}) = \frac{1}{2}\log 2\pi e Z_{n_k} - h(X_{n_k}|Y_{n_k}) \xrightarrow{\text{Pr.}}0,
\end{equation}
which implies that $h(X_{n_k}|Y_{n_k}) \xrightarrow{\text{Pr.}}\frac{1}{2}\log 2\pi ev$. Because of the uniform integrability of $\{h(X_n|Y_n),n\geq1\}$, we have $h_{n_k}\rightarrow \frac{1}{2}\log 2\pi ev$. Since the sequence $h_n$ is convergent, if we denote the limit by $h_\infty$, then $v = \exp(2h_\infty-1)/\pi$. In other word, every convergent subsequence of $V_n$ tends to the same limit $\exp(2h_\infty-1)/\pi$. Finally, since $V_n \leq l(0)$ are bounded, every subsequence of $V_n$ has a further convergent subsequence, this shows that
\begin{equation}\label{h V limit}
\lim\limits_{n\rightarrow\infty}V_n =  \exp(2h_\infty-1)/\pi := \sigma^2.
\end{equation}
Notice that $\sigma^2$ cannot be 0 because $V_n \geq a  >0 $. Clearly \eqref{h V limit} also implies \eqref{h limit}, and this completes the whole proof.\pe




\section{Entropic Conditional Central Limit Theorem}\label{section4}
This section we apply Main Theorem to derive the {\it entropic conditional central limit theorem}~(entropic CCLT). Suppose $(\xi,\eta)$ is a two-dimensional random variable such that for almost sure  realization $\eta=y$, the regular conditional distribution $\xi|\eta=y$ exists with finite variance and is absolutely continuous with respect to Lebesgue measure. Let $(\xi_1,\eta_1),(\xi_2,\eta_2),\cdots,(\xi_n,\eta_n)$ be \textit{i.i.d.} samples from $(\xi,\eta)$, $\bm{\eta}_n = (\eta_1,\eta_2,\cdots,\eta_n)$ be the random vector of the first $n$ samples of $\eta$, and $\bm{y}_n = (y_1,y_2,\cdots,y_n)$ be a realization of $\bm{\eta}_n$. Let $W_n = (\xi_1+\xi_2+\cdots+\xi_n)/\sqrt{n}$ be the normalized partial sum, and $\tilde{g}$ be a Gaussian density with the same mean and variance as the conditional distribution $(W_n|\bm{\eta}_n=\bm{y}_n)$. Denote the conditional density of $(W_n|\bm{\eta}_n=\bm{y}_n)$ by $f_n(x|\bm{y}_n)$. The Gaussianity of $W_n$ conditioning on $\bm{\eta}_n = \bm{y}_n$ is measured by
\begin{equation}\label{conditional Gaussianity}
D(W_n|\bm{y}_n) := \KL(f_n(x|\bm{y}_n)||\tilde{g}) = \int_\R f_n(x|\bm{y}_n)\log\frac{f_n(x|\bm{y}_n)}{\tilde{g}(x)}dx.
\end{equation}
Equivalently, $D(W_n|\bm{y}_n)$ is also given by
\begin{equation}
D(W_n|\bm{y}_n) = \KL(\tilde{W}_n(\bm{y}_n)||\mathcal{N}(0,1)|\bm{\eta}_n = \bm{y}_n),
\end{equation}
where
\begin{equation}\label{Sn'}
\tilde{W}_n(\bm{y}_n) = \frac{\sum\limits_{k=1}^n (\xi_k - \E(\xi|\eta_k=y_k))}{\sqrt{\sum\limits_{k=1}^n \V(\xi|\eta_k = y_k)}}
\end{equation}
stands for the conditional normalized partial sum. Notice that $D(W_n|\bm{y}_n)$ is a measurable function of $\bm{y}_n$ and hence $D(W_n|\bm{\eta}_n)$ is a random variable representing the divergence under a random realization. In this paper, it is shown that $D(W_n|\bm{\eta}_n)$ converges to 0 in expectation under some regular conditions. The specific result is stated in Theorem \ref{sum theorem}.

\begin{theorem} [Entropic CCLT]\label{sum theorem}
	
	Let $(\xi,\eta)$ be a two-dimensional random variable. Suppose $(\xi_1,\eta_1),(\xi_2,\eta_2),\cdots,(\xi_n,\eta_n)$ are independent and identically distributed as $(\xi,\eta)$. Let $W_n = (\xi_1 + \xi_2 +\cdots+\xi_n) / \sqrt{n}$, and $\bm{\eta}_n = (\eta_1,\eta_2,\cdots,\eta_n)$. Assume the following conditions hold.
	\begin{enumerate}
		\item $\V(\xi)<\infty, (\xi|\eta) \prec \mathcal{L}, \P_\eta-a.s.$, where $\mathcal{L}$ denotes Lebesgue measure and $\mu \prec \mathcal{L}$ means $\mu$ is absolutely continuous with respect to $\mathcal{L}$.
		\item  $\E J(\xi|\eta) < \infty,\  \E h(\xi|\eta)>-\infty$ (see Section \ref{section2} for the definition).
		\item (Uniform lower bound on conditional variance) There exists $a > 0$ such that $\V(\xi|\eta) \geq a,\ \P_\eta-a.s.$
	\end{enumerate}
	Denote $h_n = \E h(W_n|\bm{\eta}_n)$ and $\sigma^2 = \E\V(\xi|\eta)$, then we have
	\begin{equation}\label{main result}
	\lim\limits_{n\rightarrow\infty} h_n = \frac{1}{2}\log\left(2\pi e \sigma^2\right).
	\end{equation}
\end{theorem}

We make some remarks on the assumptions in Theorem \ref{sum theorem}. The assumption 1 and 2 are fundamental regular conditions ensuring that  the conditional distribution and related quantities are well-defined. We require uniform lower bound on conditional variance in assumption 3 to exclude some singular points of the conditional distribution. For example, let $\eta\sim \text{Uniform}(-1,1)$ and $\xi|\eta\sim\text{Uniform}(0,|\eta|)$, then $\V(\xi|y) = y^2/12$ vanishes and $h(\xi|y) = \log |y|$ tends to $-\infty$ as $y\rightarrow 0$. In this case $y=0$ is a singular point, which makes the behaviour of $D(W_n|\bm{y}_n)$ hard to analyze for some extreme realizations $\bm{y}_n$.  

 The following corollary can be deduced directly.
\begin{corollary}\label{sum corollary}
	Under the same conditions in Theorem \ref{sum theorem},
	\begin{equation}\label{main result_a}
	\lim\limits_{n\rightarrow\infty}\E D(W_n|\bm{\eta}_n) = 0.
	\end{equation}
\end{corollary}
\pf
Note that from Jensen's inequality we have
\begin{equation}
\begin{aligned}
\E D(W_n|\bm{\eta}_n) & = \frac{1}{2}\E\log \left(2\pi e \V(W_n|\bm{\eta}_n)\right) - h_n \\
&\leq \frac{1}{2}\log \left(2\pi e \E\left[\V(W_n|\bm{\eta}_n)\right]\right) - h_n\\
& = \frac{1}{2}\log \left(2\pi e \sigma^2\right) - h_n.
\end{aligned}
\end{equation}
Use Theorem \ref{sum theorem} and then completes the proof.
\e

This indicates that the KL-divergence between  the normalized conditional distribution $\tilde{W}_n(\bm{y}_n)$ and standard Gaussian  converges to 0 in probability,  as $n\rightarrow \infty$.

It is clear that the entropic CLT \eqref{Barron_KL} is covered by Theorem \ref{sum theorem} if we take $\eta$ to be independent with $\xi$. We note here that in \cite{Rubshtein1996} only the weak convergence of conditional distribution $\tilde{W}_n(\bm{y}_n)$ is obtained, however, the entropic method used in our paper leads to the strong convergence in the sense of KL-divergence as stated in Theorem \ref{sum theorem} and \eqref{main result_a}.

\subsection{Proof of Theorem \ref{sum theorem}}
In this part we use Theorem \ref{main theorem} to prove  Theorem \ref{sum theorem}. First, we show that the sequence of $nh_n=n\E h(W_n|\bm{\eta}_n)$ is superadditive, \textit{i.e.}, $(m+n)h_{m+n} \geq  mh_m + nh_n $ for all $m,n$. In fact, let $W_i^j = \frac{\xi_i+\xi_{i+1}+\cdots+\xi_j}{\sqrt{j-i+1}}$ and $\bm{\eta}_i^j = (y_i,y_{i+1},\cdots,y_j)$, from Corollary \ref{entropy power} we obtain
\begin{equation}
\begin{aligned}
h\left(W_{n+m} | \bm{\eta}_{n+m}\right)&= h\left(\frac{\sqrt{n}}{\sqrt{m+n}}W_n + \frac{\sqrt{m}}{\sqrt{m+n}}W_{n+1}^{n+m}| \bm{\eta}_{n+m}\right)\\
& \geq \frac{n}{m+n}h\left(W_n|\bm{\eta}_{n+m}\right) + \frac{m}{m+n}h\left(W_{n+1}^{n+m}|\bm{\eta}_{n+m}\right)\\
& = \frac{n}{m+n}h\left(W_n|\bm{\eta}_n\right) + \frac{m}{m+n}h\left(W_{n+1}^{n+m}|\bm{\eta}_{n+1}^{n+m}\right).
\end{aligned}
\end{equation}
Taking expectation on both sides we have
\begin{equation}
\begin{aligned}
\E h\left(W_{n+m} | \bm{\eta}_{n+m}\right) &\geq \frac{n}{m+n}\E h\left(W_n|\bm{\eta}_n\right)+\frac{m}{m+n}\E h\left(W_{n+1}^{n+m}|\bm{\eta}_{n+1}^{n+m}\right)\\
&= \frac{n}{m+n}\E h\left(W_n|\bm{\eta}_n\right)+\frac{m}{m+n}\E h\left(W_m|\bm{\eta}_m\right),
\end{aligned}
\end{equation}
which shows the superadditivity of $nh_n$. It is well known that $\frac{a_n}{n} \xrightarrow{n\rightarrow\infty} \sup\limits_n \frac{a_n}{n}$ if the sequence $\{a_n,n\geq 1\}$ is superadditive. As a result, we have
\begin{equation}\label{hn limit}
\lim\limits_{n\rightarrow\infty}h_n = \sup\limits_n h_n.
\end{equation}
Clearly, with \eqref{hn limit} in hand, it is enough to prove Theorem \ref{main theorem} for $n = 2^k$. As explained in \eqref{increase on h_Sn}, the reason to take $n$ as power of 2 is that the difference $h_{2n} - h_n$ can be written as the form of entropy jump:
\begin{equation}\label{hn entropy jump}
h_{2n} - h_n= \E\left[h\left(\frac{W_n + W'_n}{\sqrt{2}}| \bm{\eta}_{2n}\right) - \frac{h(W_n|\bm{\eta}_n) + h(W'_n|\bm{\eta}'_n)}{2}\right],
\end{equation}
where $W_n'$ and $\bm{\eta}'_n$ are independent copies of $W_n$ and $\bm{\eta}_n$, respectively. We note here that \eqref{hn limit} and \eqref{hn entropy jump} are the only properties of differential entropy used in our proof, which are analog to \eqref{entropy condition 1} and \eqref{entropy condition 2}.

\textbf{Proof of Theorem \ref{sum theorem}.} Let $U_n = W_{2^n}, T_n = \bm{\eta}_{2^n}$. We first show that $(U_n,T_n)$ satisfies the assumptions in Theorem \ref{main theorem}. Firstly,
\begin{equation}
\V(U_n|T_n) = \V\left(\frac{\xi_1 + \cdots \xi_{2^n}}{\sqrt{2^n}}|\bm{\eta}_{2^n}\right) = \frac{\V(\xi|\eta_1)+\cdots +\V(\xi|\eta_{2^n})}{2^n} \geq a.
\end{equation}
Besides, from Lemma \ref{Fisher ineq} we obtain
\begin{equation}
\begin{aligned}
J(U_n|\eta_n) & = J\left(\frac{\xi_1 + \cdots \xi_{2^n}}{\sqrt{2^n}}|\bm{\eta}_{2^n}\right) \leq \frac{1}{2^n}\sum\limits_{k=1}^{2^n}J\left(\xi_k|\bm{\eta}_{2^n}\right) = \frac{1}{2^n}\sum\limits_{k=1}^{2^n}J\left(\xi_k|\eta_k\right)
\end{aligned}
\end{equation}
Taking expectation on both sides we have
\begin{equation}
\E J(U_n|T_n)\leq J(\xi|\eta) <\infty.
\end{equation}
In addition, it is not hard to verify that
\begin{equation}
\begin{aligned}
\E[L_{U_n|T_n}(R)] = \E\left[\bar{W}_n^2\bm{1}_{\{|\bar{W}_n|\geq R\}}\right],
\end{aligned}
\end{equation}
where $\bar{W}_n$ is defined as
$$\bar{W}_n = \frac{\sum\limits_{k=1}^n (\xi_k - \E(\xi_k|\eta_k))}{\sqrt{n}}.$$
Let
\begin{equation}
l(R) = \sup\limits_n \E(\bar{W}_n^2\bm{1}_{|\bar{W}_n|\geq R}),
\end{equation}
since $\{\bar{W}_n^2,n\geq1\}$ is uniformly integrable(see proof in Appendix \ref{proof of 2th moment}),
$l(R)$ decreases to 0 as $R\rightarrow\infty$.

Clearly, \eqref{entropy condition 1} and \eqref{entropy condition 2} follow from \eqref{hn limit} and \eqref{hn entropy jump}. All of the discussion above show that $(U_n,T_n)$ satisfies all conditions of Theorem \ref{main theorem}. Consequently, from \eqref{general D} we conclude that
\begin{equation}\label{D Un tn}
D(U_n|T_n) \xrightarrow{\text{Pr.}} 0,\ n\rightarrow\infty.
\end{equation}
Let $N = 2^n$, $h_N(\bm{\eta}) = h(W_N|\bm{\eta}_N), V_N(\bm{\eta}) = \V(W_N|\bm{\eta}_N)$, then \eqref{D Un tn} is equivalent to
\begin{equation}
\frac{1}{2}\log 2\pi e V_N(\bm{\eta}) - h_N(\bm{\eta}) \xrightarrow{\text{Pr.}} 0.
\end{equation}
Notice that $V_N(\bm{\eta}) = \frac{\V(\xi|\eta_1)+\cdots+\V(\xi|\eta_N)}{N}$, the law of large numbers leads to $V_N(\bm{\eta})\xrightarrow{\text{Pr.}} \E\V(\xi|\eta) = \sigma^2$, which implies that
\begin{equation}
h_N(\bm{\eta})\xrightarrow{\text{Pr.}}\frac{1}{2}\log \left(2\pi e\sigma^2\right).
\end{equation}
By the entropy jump inequality \eqref{EJI},
\begin{equation}
h_N(\bm{\eta}) \geq \frac{h(\xi|\eta_1) + \cdots + h(\xi|\eta_N)}{N} := H_N.
\end{equation}
Using the law of large numbers again we have $H_N \xrightarrow{a.s.} \E h(\xi|\eta) > -\infty$. Now by Fatou's lemma,
\begin{equation}\label{liminf}
\E\left[\liminf\limits_{N\rightarrow \infty} (h_N(\bm{\eta})-H_N)\right] \leq \liminf\limits_{N\rightarrow \infty}\E\left[h_N(\bm{\eta})-H_N\right] = \liminf\limits_{N\rightarrow \infty}h_N - \E h(\xi|\eta).
\end{equation}
Without loss of generality, we can assume $h_N(\bm{\eta})\xrightarrow{a.s.}\frac{1}{2}\log \left(2\pi e\sigma^2\right)$ (or we just take an almost surely convergent subsequence), then
\begin{equation}
\liminf\limits_{N\rightarrow \infty} (h_N(\bm{\eta})-H_N) = \frac{1}{2}\log\left(2\pi e\sigma^2\right) - \E h(\xi|\eta).
\end{equation}
Therefore, from \eqref{liminf} we have $\liminf\limits_{N\rightarrow \infty}h_N \geq \frac{1}{2}\log 2\pi e\sigma^2$. On the other hand, by Jensen's inequality,
\begin{equation}
h_N \leq \frac{1}{2}\E\log\left(2\pi e V_N(\bm{\eta})\right) \leq \frac{1}{2}\log\left(2\pi e \E V_N(\bm{\eta})\right) = \frac{1}{2}\log 2\pi e\sigma^2,
\end{equation}
which shows that \eqref{main result} holds for $N=2^n$. As discussed at the beginning of this section, it is enough to imply Theorem \ref{main theorem} since $nh_n$ is superadditive.\pe

\subsection{Rate of Convergence with Independent Component}
We can establish the rate of convergence for the limit theorem. Indeed,  we note that there is a further form for \eqref{main to sum1}
\begin{equation}\label{independent sum form}
(\xi_n|\eta_n=y)=(\frac{\xi_0^1+\xi_0^2+\cdots+\xi_0^{2^n}}{2^{n/2}}|\eta_0^1=y_1,\eta_0^2=y_2,\cdots ,\eta_0^{2^n}=y_{2^n})\overset{d}{=}\frac{\eta_1+\eta_2+\cdots+\eta_{2^n}}{2^{n/2}},
\end{equation}
where  $\{\eta_i, i=1,2,\cdots,2^n\}$ are independent with distribution
\begin{equation*}
\eta_i\overset{d}{=}(\xi_0|\eta_0=y_i).
\end{equation*}
Thus, Theorem\ref{main theorem} is closely related with CLT for classical independent sum. \cite[Theorem 1.2, (12.6)]{Esseenbound} gave the speed of convergence in KL-divergence sense for independent sum.
\begin{lemma}\label{Essbound}
	For independent sum $T_n=\sum_{i=1}^n \xi_i$,  $ \sum_{i=1}^n \V(\xi_i):=B_n$, if $\sup_i D(\xi_i)\le D<\infty$, then
	$$D(T_n)\le ce^{62D}L_4,$$
	where $c$ is a constant,
	$$D(\xi):=\KL(\xi\|\mathcal{N}(\E \xi, \V(\xi))),$$
	$$L_s:=\sum_{i=1}^n \E |\xi_i-\E \xi_i|^s/B_n^{s/2}.$$
\end{lemma}
Since  Lemma \ref{Essbound} needs an upper bound for $D(\xi_i)$, we shall use another lemma   \cite[(10)]{S-EPI} to go through it.
\begin{lemma}\label{control D by VJ}
	Given $\xi$, $\E \xi=0, \V(\xi)<\infty$, then
	\begin{equation}
	D(\xi)\le \frac{1}{2}\log(\V(\xi) J(\xi)).
	\end{equation}
\end{lemma}

Thanks to Lemma \ref{Essbound} and Lemma \ref{control D by VJ}, we can deduce the rate theorem as follows.
\begin{theorem}\label{rate theorem}
	Let $(\xi,\eta)$ be a two-dimensional random variable, and suppose that  $(\xi_1,\eta_1),(\xi_2,\eta_2),\cdots,(\xi_n,\eta_n)$ are independent and identically distributed as $(\xi,\eta)$. Let $W_n = (\xi_1 + \xi_2 +\cdots+\xi_n) / \sqrt{n}$, and $\bm{\eta}_n = (\eta_1,\eta_2,\cdots,\eta_n)$. Assume the following conditions hold.
	\begin{enumerate}
		\item $(\xi|\eta) \prec \mathcal{L}, \ \P_\eta-a.s.$, where $\mathcal{L}$ denotes Lebesgue measure and $\mu \prec \mathcal{L}$ means $\mu$ is absolutely continuous with respect to $\mathcal{L}$ .   Further, for $\P_\eta-a.s.$, the conditional random variable $(\xi|\eta)$ has absolutely continuous density.
		\item (Finite $s$th moment for variance and Fisher information)  There   exists $s>64$ s.t.
		\begin{equation}
		\E (\V(\xi|y))^s<\infty,\;\; \E (J(\xi|y))^s<\infty.
		\end{equation}
	\end{enumerate}
	Then $\forall \varepsilon>0$,
	\begin{equation}\label{rate result}
	\E D(W_n|\bm{\eta}_n)=O\left(\frac{1}{n^{\frac{s-64}{s+64}-\varepsilon}} \right).
	\end{equation}
\end{theorem}
\pf  For any realization $\bm{\eta}_n=\bm{y}_n=(y_1,y_2,\cdots,y_n)$, we can write the conditional distribution as
\begin{equation}
(W_n|\bm{\eta}_n=\bm{y}_n)\overset{d}{=}\eta_1+\eta_2+\cdots+\eta_{n},
\end{equation}
where $\eta_i$s are independent and
$$\eta_i\overset{d}{=}\frac{(\xi|\eta=y_i)}{\sqrt{n}}.$$
First we control the upper bound for $D(\eta_i)$ in order to satisfy the uniformly upper bound condition for $D$ in Lemma \ref{Essbound}. In fact by Lemma \ref{control D by VJ}
\begin{equation}
\begin{aligned}
D(\eta_i)=\frac{1}{2}\log(\V(\xi|\eta=y_i)J(\xi|\eta=y_i)), \;\forall i.
\end{aligned}
\end{equation}

Next, define
$$A(M):=\{ y:   \V(\xi|\eta=y)\le M, J(\xi|\eta=y)\le M \},$$
then by Lemma \ref{Essbound}, $\forall \bm{y}_n=(y_1,\cdots,y_n)\in (A(M))^n$,\
\begin{equation}\label{bound for low variance AM}
\begin{aligned}
D(W_n|\bm{\eta}_n=\bm{y}_n)&\le cM^{62}\dfrac{\sum_{i=1}^{n}|\eta_i-\E\eta_i|^4}{(\sum_{i=1}^{n}\V(\eta_i))^2}\\
&=cM^{62}\dfrac{\sum_{i=1}^{n}|\eta_i|^4+(\E\eta_i)^4}{(\sum_{i=1}^{n}\V(\eta_i))^2}\\
&\le cM^{64}\sum_{i=1}^{n}|\eta_i|^4+(\E\eta_i)^4,
\end{aligned}
\end{equation}
where the last inequality holds since $\V(\xi)J(\xi)\ge 1$. On the other hand,
use condition 4 and Chebyshev inequality, we obtain
\begin{equation}\label{pro bound for AM}
\P(A(M))\ge 1-\dfrac{\E (\V(\xi|y))^s+ \E (J(\xi|y))^s}{M^s}.
\end{equation}
In addition, for any $\bm{y}_n$, we use Lemma \ref{control D by VJ} and Lemma 1.3 in \cite{Carlen1991} to obtain $\forall t>0$,
\begin{equation}\label{control D by v,J}
\begin{aligned}
D(W_n|\bm{\eta}_n)=\bm{y}_n)&C(t)(\V^t(W_n|\bm{\eta}_n)+J^t(W_n|\bm{\eta}_n))\\
&\le C(t)(\V^t(W_n|\bm{\eta}_n)+\left( \frac{1}{n}\sum_{i=1}^nJ(\xi_i|y_i) \right)^t).
\end{aligned}
\end{equation}

Combine the above results, we have for all $t\in(0,1),$
\begin{align}
\E D(W_n|\bm{\eta}_n)=\E D(W_n|\bm{\eta}_n=\bm{y}_n)&1_{\bm{y}_n\in (A(M))^n}+\E D(W_n|\bm{\eta}_n=\bm{y}_n)1_{\bm{y}_n\notin (A(M))^n}\\
\overset{(a)}{\le } \E cM^{64}(\sum_{i=1}^{n}|\eta_i(y_i)|^4+(\E\eta_i)^4)&1_{\bm{y}_n\in (A(M))^n}+\E C(t)(\V^t(W_n|\bm{\eta}_n)\\+
&\left( \frac{1}{n}\sum_{i=1}^nJ(\xi_i|y_i) \right)^t)1_{\bm{y}_n\notin A(M)} \nonumber \\
\overset{(b)}{\le }   \E cM^{64}(\sum_{i=1}^{n}|\eta_i(y_i)|^4+(\E\eta_i)^4)&+C(t)\left(\V(\xi)+\E J(\xi|\eta)\right)^t\P^{1-t}\left( \bm{y}_n\notin A(M)\right)\\
\overset{(c)}{\le }   cM^{64}\dfrac{\E|\xi|^4}{n}+C&\left(1-(1-\dfrac{\E (\V(\xi|y))^s+ \E (J(\xi|y))^s}{M^s})^n\right)^{1-t}\\
\overset{(d)}{\le }  \dfrac{C_1M^{64}}{n}+&C_2\left(\frac{n}{M^s} \right)^{1-t}.
\end{align}
where $(a)$ holds due to \eqref{bound for low variance AM} and \eqref{control D by v,J}; $(b)$ comes from Cauchy inequality; $(c)$ is deduced by the definition of $\eta_i$, Jensen's inequality and \eqref{pro bound for AM}; $(d)$ is from Taylor expansion of $x^n$ around $x=1$. Finally, take $M=n^{-\frac{2-t}{62+(1-t)s}}$, then the proof is completed.

\e

\section{Asymptotic Gaussianity Measured by Fisher infomation} \label{section5}
In this section, we investigate the asymptotic Gaussianity measured by Fisher information. We  consider the case where $X$ is perturbed by Gaussian, that is , $X$ is  decomposed into the sum of $X_0$ and an independent Gaussian $G_a$, as  \eqref{IGD-n}, and deduce the asymptotic Gaussianity of $W_n$ by means of Fisher information.
\begin{theorem}\label{Fisher limit}
	Under the same setting of Theorem \ref{main theorem}, if condition 3 in Theorem \ref{main theorem} holds and we further assume $X_n$
	has the form
	\begin{equation}\label{IGD-n}
	X_n=X_{0,n}+G_{a,n},
	\end{equation}
	where $G_{a,n}\sim \mathcal{N}(0,a)$   is independent with $(X_{0,n},Y_n)$, then
	\begin{equation}\label{J limit}
	\lim\limits_{n\rightarrow\infty} \E J(X_n|Y_n) = 1/\sigma^2,
	\end{equation}
	where  $\sigma^2= \lim_{n\to\infty}\E \V(X_n|Y_n)$ is the same as in \eqref{V limit}.
\end{theorem}

Taking advantage of this decomposition, we claim that the assumption 1 and 2 in Theorem \ref{main theorem} are naturally satisfied. Actually,
\begin{equation}\label{130}
\V(X_n|Y_n) = \V(X_{0,n}+G_{a,n}|Y_n) \geq \V(G_{a,n}|Y_n) = a.
\end{equation}
Besides, since convolution decreases Fisher information (\textit{e.g.}, see \cite[(2.2)]{Bar1986}), we have
\begin{equation}\label{J decrease}
J(X_n|Y_n) = J(X_{0,n}+G_{a,n}|Y_n) \leq J(G_{a,n}|Y_n) = \frac{1}{a}.
\end{equation}
In addition, \eqref{IGD-n} also implies that
\begin{equation}
h(X_n|Y_n) = h(X_{0,n}+G_{a,n}|Y_n) \geq h(G_{a,n}|Y_n) = \frac{1}{2}\log 2\pi e a.
\end{equation}
Using this uniform lower bound and the Fatou's lemma argument in \eqref{liminf}, it is able to show \eqref{h limit} without the uniform integrability requirement of the sequence $\{h(X_n|Y_n),n\geq1\}$.

More importantly, we can establish the continuity of Fisher information under the assumption of \eqref{IGD-n}, and then deduce the asymptotic
Gaussianity by means of  Fisher information as stated in Theorem \ref{Fisher limit}. For $a>0$ and function $l(R)$, define
\begin{equation}
\mathcal{I}_{a,l}:=\{X_0 + G_a:\  L_{X_0+G_a}(R)\leq l(R),\forall R>0, \E X_0=0, X_0\text{ is independent with }G_a     \}
\end{equation}
and $\overline{\mathcal{I}_{a,l}}:=(\mathcal{I}_{a,l})^{cl}$, where the closure is taken in $L^1_{1+x^2}$. The next lemma shows that Fisher information is uniformly continuous on $\overline{\mathcal{I}_{a,l}}$ with respect to $L^1$ norm.

\begin{lemma}\label{continuity of Fisher information}
	The set $\overline{\mathcal{I}_{a,l}}$ is compact in $L^1$. In addition, if $p_n\in\overline{\mathcal{I}_{a,l}}$ and $p_n \xrightarrow{L^1} p$, then $J(p_n)\rightarrow J(p)$ as $n\rightarrow\infty$.
\end{lemma}
\pf Without loss of generality, we assume $a\in(0,1)$. For any $X\in\mathcal{I}_{a,l}$, we have
\begin{equation}\label{I decompose}
\begin{aligned}
X = X_0 + G_a &\overset{d}{=} X_0 + \hat{G}_{a/2} + G_{a/2} \overset{d}{=} X_0 + \hat{G}_{a/2} + \sqrt{a/2}G\\
& = e^{-t} Z + \sqrt{1-e^{-2t}}G = p_t^* Z,
\end{aligned}
\end{equation}
where $t = -\frac{\log(1-a/2)}{2}$ and $Z = e^{t}(X_0 + \hat{G}_{a/2})$. Clearly $\E Z = 0$. Since
\begin{equation}
\Gamma:=\{|X_0|\geq 2R\}\cap\{|G_a|\leq R\}\subset \{|X|\geq R\},
\end{equation}
we have
\begin{equation}
\begin{aligned}
l(R) \geq L_X(R) &\geq \E[X^2\bm{1}_{\Gamma}]\geq  \E[(X_0^2+2X_0G_a)\bm{1}_{\Gamma}]\\
& \geq L_{X_0}(2R)\P(|G_a|\leq R) +2\E[X_0\bm{1}_{\{|X_0|\geq 2R\}}]\E[G_a\bm{1}_{\{|G_a|\leq R\}}]\\
& \geq L_{X_0}(2R)(1 - \frac{a^2}{R^2}),
\end{aligned}
\end{equation}
where the last inequality holds by Chebyshev's inequality and the symmetry of $G$. Therefore, the tail variance of $X_0$ is also uniformly bounded. Using the similar argument in the proof of Proposition \ref{L of HtL}, it is able to find a function $l'$ decreasing to 0 independent with $X$ such that $L_Z \leq l'$. Consequently, we have $\mathcal{I}_{a,l}\subset\mathcal{H}_{t,l'}$. Notice that $\bar{\mathcal{I}}_{a,l}$ is closed, and Lemma \ref{compactness of HtL} shows the compactness of $\mathcal{H}_{t,l'}$, this implies that $\overline{\mathcal{I}_{a,l}}$ is compact in $L^1_{1+x^2}$. Finally, since $\|\cdot\|_{L^1} \leq \|\cdot\|_{L^1_{1+x^2}}$, $\overline{\mathcal{I}_{a,l}}$ is also compact in $L^1$. This completes the proof of the first part.

Without loss of generality, suppose $X_n\in\mathcal{I}_{a,l}$ with density $p_n$ such that $\|p_n-p\|_{L^1}\rightarrow 0$ as $n\rightarrow\infty$. The proof of the second part is divided into 4 steps.

Firstly we show that we can find a subsequence $p_{n_k}$ such that $p'_{n_k}\xrightarrow{a.e.}p'$. Because $\mathcal{I}_{a,l}\subset\mathcal{H}_{t,l'}$ and $\mathcal{H}_{t,l'}$ is compact, it is able to find $q\in \mathcal{H}_{t,l'}$ and subsequence $p_{n_k}$ such that
\begin{equation}\label{pnk limit}
p_{n_k}\xrightarrow{L^1_{1+x^2}}q.
\end{equation}
Since $p_n\xrightarrow{L^1}p$, we must have $p = q, a.e.$ Write $X_n \overset{d}{=} p^*_t Z_n$ as in \eqref{I decompose}, denote the density of $e^{-t}Z_n$ by $f_n$, and the density of $\sqrt{1-e^{-2t}}G$ by $\phi_t$.Then \eqref{pnk limit} implies that
\begin{equation}\label{fnk * limit}
f_{n_k}*\phi_t \xrightarrow{L^1_{1+x^2}}p.
\end{equation}
Notice that $e^{-t}Z_{n_k} = X_{0,n_k} + \hat{G}_{a/2}$ also takes the form of independent Gaussian decomposition, from the first part of this proof we conclude that $e^{-t}Z_{n_k}\in \mathcal{H}_{\tilde{t},\tilde{l}}$ for some $\tilde{t}>0$ and function $\tilde{l}$ decreasing to 0. Using the compactness of $\mathcal{H}_{\tilde{t},\tilde{l}}$ again we can find a further subsequence such that
\begin{equation}\label{fnkj limit}
f_{n_{k_j}} \xrightarrow{L^1_{1+x^2}} f
\end{equation}
for some $f \in \mathcal{H}_{\tilde{t},\tilde{l}}$. Since $\forall y\in \R$,
\begin{equation}
\left|(f_{n_{k_j}} - f)*\phi_t(y)\right| = \int_\R \left|f_{n_{k_j}}(x) - f(x)\right|\phi_t(y-x)dx \leq C_t \|f_{n_{k_j}}-f\|_{L^1} \rightarrow 0,\ n\rightarrow\infty,
\end{equation}
we have $f_{n_{k_j}}*\phi_t \xrightarrow {a.e.} f*\phi_t$. From \eqref{fnk * limit} we obtain that $p = f\ast \phi_t,\ a.e.$ Since
\begin{equation}
\begin{aligned}
|p'_{n_{k_j}}(y)-p'(y)|&=\left|\int_\R \frac{y-x}{1-e^{-2t}}(f_{n_{k_j}}(x)-f(x))\phi_t(y-x)dx\right|\\
&\leq  C'_t|y|\int_\R |f_{n_{k_j}}(x)-f(x)|dx+ C'_t\int_R |x| |f_{n_{k_j}}(x)-f(x)|dx\\
&\leq C'_t|y|\|f_{n_{k_j}}-f\|_{L^1} + C'_t\|f_{n_{k_j}}-f\|_{L^1_{1+x^2}},\ \forall y\in\R,
\end{aligned}
\end{equation}
using \eqref{fnkj limit} we conclude that $p'_{n_{k_j}}\xrightarrow{a.e.}p'$.

Secondly, we show that it is able to take a subsequence such that $(\sqrt{p_{n_k}})'\xrightarrow{a.e.}(\sqrt{p})'$. In fact, since $p_n\xrightarrow{L^1}p$, we can take subsequence such that $p_{n_k}\xrightarrow{a.e.} p$. From step 1, it is able to take a further subsequence such that $p'_{n_{k_j}}\xrightarrow{a.e.} p'$. Since $(\sqrt{p})' = p'/(2\sqrt{p})$, we have $(\sqrt{p_{n_{k_j}}})'\xrightarrow{a.e.}(\sqrt{p})'$.

Thirdly, let $p_{n_k}$ be the subsequence taken in the way above, we show $(\sqrt{p_{n_k}})'$ also converges in the sense of $L^2$ norm. Since $X_n \overset{d}{=} p_t^* Z_n\in\mathcal{H}_{t,l'}$, from Proposition \ref{bound on H and B} we obtain
\begin{equation}
\left((\sqrt{p_{n_k}})'\right)^2 = \frac{(p'_{n_k})^2}{4p_{n_k}}\leq B_t p^*_{2t}(e^t Z_{n_k}),
\end{equation}
where $B_t > 0$ is a constant. By some simple calculations we have $p^*_{2t}(e^t Z_{n_k}) = f_{n_k} * \phi_{2t}$. Using \eqref{fnkj limit} we obtain a further subsequence $f_{n_{k_j}} * \phi_{2t} \xrightarrow{L^1} f*\phi_{2t}$. Now from dominated convergence theorem we conclude that
\begin{equation}
(\sqrt{p_{n_{k_j}}})' \xrightarrow{L^2} (\sqrt{p})'.
\end{equation}

Finally, since $J(p_{n_{k_j}}) = 4\|(\sqrt{p_{n_{k_j}}})'\|_{L^2}^2$, we have $J(p_{n_{k_j}}) \rightarrow J(p)$.

The above 4-steps discussion implies that for any subsequence $J(p_{n_k})$, we can find a further subsequence such that $J(p_{n_{k_j}})\rightarrow J(p)$, this indicates that $J(p_n)\rightarrow J(p)$ as $n\rightarrow\infty$, which completes the whole proof.\e

Using Lemma \ref{continuity of Fisher information}, we show that the conditional Fisher information $\E J(X_n|Y_n)$ also converges to a limit that is coincides with Gaussian, \textit{i.e.} Theorem \ref{Fisher limit}.

\textbf{Proof of Theorem \ref{Fisher limit}.}   Denote $J_n=\E J(X_n|Y_n), V_n=\E\V(X_n|Y_n)$.  On the one hand, by the CR bound \eqref{CR bound} and Cauchy-Schwartz inequality,
\begin{equation}\label{ge}
V_nJ_n\geq \left(\E\sqrt{J(X_n|Y_n)\V(X_n|Y_n)}\right)^2\geq 1,\ \ \forall n,
\end{equation}
which implies $\liminf\limits_{n\rightarrow\infty} J_n \geq \frac{1}{\sigma^2}$.

On the other hand, for any $\alpha>0$, by Proposition \ref{F and L condition}, we can find a function $l$ decreasing to 0 such that the set
\begin{equation}
U_{n}=\{ y_n:\ L_{X_n|y_n}\leq l  \}
\end{equation}
has probability $\P(U_n)\geq 1- \alpha$. From Lemma \ref{continuity of Fisher information} we see that the Fisher information is uniformly continuous on $\mathcal{I}_{a,l}$ with respect to $L^1$ norm. Using Pinsker's inequality, we conclude that for any $\epsilon>0$, there exists $\delta = \delta(\epsilon,a,L)>0$ such that for any $p,q \in \mathcal{I}_{a,l}$,
\begin{equation}\label{KL-J continuity}
\KL(p||q) \leq \delta \Longrightarrow |J(p) - J(q)| \leq \epsilon.
\end{equation}
Now define
\begin{align}
&R_n = \{y_n:\ D(X_n|y_n) \leq \delta\},\\
&W_n = \{y_n:\ |\V(X_n|Y_n)-\sigma^2|\leq \epsilon\}.
\end{align}
Let $T_n = U_n \cap R_n \cap W_n$, since $D(X_n|y_n)$ is the divergence between $(X_n|y_n)$ and the Gaussian with the same variance, from \eqref{KL-J continuity} we obtain
\begin{equation}
\left|J(X_n|y_n) - \frac{1}{\V(X_n|y_n)}\right| \leq \epsilon,\ \forall y_n \in T_n.
\end{equation}
As a result,
\begin{equation}
\begin{aligned}
J_n & = \E\left[(J(X_n|Y_n)\bm{1}_{T_n}\right]+\E\left[(J(X_n|Y_n)\bm{1}_{T_n^C}\right]\\
&\leq \E\left[\left(\epsilon + \frac{1}{\V(X_n|Y_n)}\right)\bm{1}_{T_n}\right] + \frac{1}{a}\P(T_n^C),
\end{aligned}
\end{equation}
where we use \eqref{J decrease} in the last inequality. Notice that on $T_n$ we have
\begin{equation}
\sigma^2-\epsilon \leq \V(X_n|Y_n) \leq \sigma^2 + \epsilon,
\end{equation}
which implies
\begin{equation}
\begin{aligned}
J_n &\leq \epsilon + \frac{1}{\sigma^2-\epsilon} + \frac{\P(U_n^C)+\P(R_n^C)+\P(W_n^C)}{a}\\
&\leq \epsilon + \frac{1}{\sigma^2-\epsilon} + \frac{\alpha+\P(R_n^C)+\P(W_n^C)}{a}.
\end{aligned}
\end{equation}
By \eqref{general D} and \eqref{V limit} in Theorem \ref{main theorem}, $\P(R_n^C),\P(W_n^C)\rightarrow 0$ as $n\rightarrow\infty$, thus
\begin{equation}
\limsup\limits_{n\rightarrow\infty} J_n \leq \epsilon + \frac{1}{\sigma^2-\epsilon} + \frac{\alpha}{a}.
\end{equation}
Since $\alpha$ and $\epsilon$ can be arbitrary small, we obtain $\limsup\limits_{n\rightarrow\infty} J_n \leq \frac{1}{\sigma^2}$. This completes the whole proof. \pe

\section{Hadamard Compression for Continuous Input}\label{section6}

In this section we investigate the Hadamard compression for continuous input. We recall first some notations which we have introduced in Section \ref{section1.3}.  Consider the n-iteration of Hadamard transform
$H_{\lambda, n}=B_N(H_\lambda)^{\otimes n}.$  We suppose that the input $\xi^n=(\xi^n_1,\xi^n_2,\cdots, \xi^n_N)^T$ satisfies $\{\xi^n_i\}_{i=1}^N\overset{\textit{i.i.d.}}{\sim} \xi$, where $\xi$ has density $p(x)$. Let  $\eta^n=(\eta^n_1,\eta^n_2,\cdots, \eta^n_N)^T=H_{\lambda,n}\xi^n$ be the received signal.
For $i<j,$ we write $\eta_{i:j}^n$ for $ (\eta^n_i, \eta_{i+1}^n,\cdots, \eta^n_j).$ We denote the density of $\eta_k^n$ as $f_k^n$, the joint density of $\eta_{i:j}^n=(\eta^n_i, \eta_{i+1}^n,\cdots, \eta^n_j)$ as $f_{i:j}^n$, and the conditional density of $\eta_{k}^n$ given  $\eta_{1:k-1}^n=(y_1, y_2, \cdots, y_{k-1})$ as $f^n_{k|y}$ with $y$ stands for $(y_1, y_2, \cdots, y_{k-1})$.  Let $\Omega:=\{0,1\}^\infty$ be the polarization path space. For each $n$-iteration polarization path $(\omega_1,\omega_2,\cdots\omega_n),$ we define	$k(n,\omega)=[\omega_1\omega_2\cdots\omega_n]_{2\to 10}+1,$ where $[\cdot]_{2\to10}$ is the conversion from  binary to decimal. Then, due to the action of  the bit permutation matrix  $B_N,$ $\eta^n_k(\omega):=\eta^n_{k(n,\omega)}$ is exactly the $k$th component $\eta^n_k$ of $ \eta^n.$

Recall that in section \ref{section3} we have studied the behaviour of a special path $\omega=(0,0,\cdots,0),$  in this section we shall study the behaviour of other polarization pathes. To this end, from now on we  impose a probability measure $\P_\Omega$ on the path space $\Omega:=\{0,1\}^\infty$ in such a way  that all the pathes appear with equal probability. More precisely,
let $\mathcal{F}$ be the $\sigma$-algebra generated by all cylindrical sets $S(b_1,b_2,\cdots,b_n):=\{ \omega\in\Omega: (\omega_1,\omega_2,\cdots,\omega_n)=(b_1,b_2,\cdots,b_n) \}, (b_1,b_2,\cdots,b_n)\in\{0,1\}^n,  n\ge 1.$ We impose $\P_\Omega$ on $(\Omega, \mathcal{F} )$ in such a way that  $\P_\Omega(S(b_1,b_2,\cdots,b_n))=1/2^n$  for each cylindrical set . The existence of such $\P_\Omega$ is ensured by the well known Kolmogorov extension theorem.
In the following we define three random sequences on the probability space $(\Omega, \mathcal{F}, \P_\Omega )$ which will play central roles in our forthcoming study.

\begin{definition}\label{three process}
	On the probability space $(\Omega, \mathcal{F}, \P_\Omega),$ we define {\it differential entropy sequence}  $\{h_n\}_{n=1}^\infty$, {\it variance sequence}  $\{V_n\}_{n=1}^\infty$, and {\it Fisher information sequence } $\{J_n\}_{n=1}^\infty$  as follows.
	\begin{align*}
	& h_n(\omega) = \int_{\R^{k-1}} f_{1:k(n,\omega)-1}(y)h\left(f_{k(n,\omega)|y}\right)dy = \E[ h\left(\eta^n_{k(n,\omega)}|\eta^n_{1:k(n,\omega)-1}\right)];\\
	& V_n(\omega) = \int_{\R^{k-1}} f_{1:k(n,\omega)-1}(y)\V\left(f_{k(n,\omega)|y}\right)dy = \E[\V\left(\eta^n_{k(n,\omega)}|\eta^n_{1:k(n,\omega)-1}\right)];\\
	& J_n(\omega) = \int_{\R^{k-1}} f_{1:k(n,\omega)-1}(y)J\left(f_{k(n,\omega)|y}\right)dy = \E[ J\left(\eta^n_{k(n,\omega)}|\eta^n_{1:k(n,\omega)-1}\right)].
	\end{align*}
	
\end{definition}

To investigate Hadamard compression of continuous random inputs, below we assume that the input signal  $\xi$ is the same as in Chapter 4 of \cite{ULE}, that is,
\begin{equation}\label{IGD1}
	\xi = \hat{\xi}_0 + G_a
\end{equation}
for some $a>0$, where $G_a$ is a Gaussian random variable with variance $a$ that is independent with $\hat{\xi}$. Note that the above form of random variables is a classical form in entropy analysis which aften appears in practice (see e.g. \cite{ULE},\cite{Bar1986}, \cite{EI-CLT}, \cite{Costa1985} and \cite{McKean1966}).

\subsection{Limit distributions of $a.s.$ polarization pathes}

The theorem below  shows that when input distribution has form \eqref{IGD1}, then under Hadamard transform and SC decoding frame, almost every distribution of the signal conditional on the knowledge of the previous signals   tends to Gaussian, when the  block length tends to infinity.
\begin{theorem}\label{polar theorem}
	If the input $\xi$ has form \eqref{IGD1} with $\E\hat{\xi}_0=0$ and $ \E\hat{\xi}^4_0<\infty$, then there exists $ O\subset\Omega, \P_\Omega(O)=1$, such that
	\begin{equation}\label{polarhn}
		h_\infty(\omega)=\frac{1}{2}\log(2\pi e V_\infty(\omega)),\;\;\forall \omega\in O;
	\end{equation} and hence
	\begin{equation}\label{polarhn1}
\lim\limits_{n\rightarrow\infty}\E D(\eta^n_{k(n,\omega)}|\eta^n_{1:k(n,\omega)-1}) = 0,\;\;\forall \omega\in O;
	\end{equation}
	\begin{equation}\label{polarVn}
		\lim_{n\to\infty} \E| \V(\eta^n_k(\omega)|\eta^n_{1:k(\omega)-1})-V_\infty(\omega) |=0,\;\;\forall \omega\in O;
	\end{equation}
	where $h_\infty(\omega)=\lim_{n\to\infty} h_n(\omega), V_\infty(\omega)=\lim_{n\to\infty} V_n(\omega).$
\end{theorem}
\pf

In accordance with the notations we introduced in \eqref{Xn, Yn}, we write	$X_n=\eta^n_{k(n,\omega)}$ and $ Y_n =\eta^n_{1:k(n,\omega)-1}: =(\eta^n_{1},\eta^n_{2},\cdots,\eta^n_{k(n,\omega)-1})$. We have shown that the assumption \eqref{entropy condition 2} of $h(X_n|Y_n)$ in Theorem \ref{main theorem}, i.e. the entropic property  \eqref{basic1}, holds  for all $\omega \in \Omega$ (see Appendix \ref{A.1}). When the input signal  $\xi$ takes the form \eqref{IGD1},
we will find that for  $a.s. \omega \in \Omega,$
\begin{equation}\label{basic2}
	\mbox{sequence}~ \{ h(X_n(\omega)|Y_n(\omega))\}_{n=1}^\infty~ \mbox{is  uniformly integrable},
\end{equation}
and
\begin{equation}\label{basic3}
	\lim_{n\to\infty}   \E_{Y_n} h(X_n|Y_n)(\omega)~ \mbox{exists}.
\end{equation}
Thus the assumptions that  the uniformly integrability of $h_n$ and \eqref{entropy condition 1}  in Theorem \ref{main theorem} are held for $(X_n, Y_n)(\omega)=(\eta^n_{k(n,\omega)}, \eta^n_{1:k(n,\omega)-1})$, \textit{a.s.}
$\omega$. The detail proofs of the properties  \eqref{basic2} and \eqref{basic3} will be given in Appendix \ref{A.2}.  Naturally, we can choose $O\subset \Omega$ to be the set of all $\omega$'s that ensure \eqref{basic2} and \eqref{basic3}  hold. Clearly, $\P_\Omega(O)=1$.

 We now check Conditions 1-3 listed in Theorem \ref{main theorem}.  Since
$$\eta_k^n=[H_{\lambda,n} \xi^n  ]_k=[H_{\lambda,n} (\hat{\xi}_0^n+G_a^n)  ]_k=[H_{\lambda,n}\hat{\xi}_0^n+G_a^n  ]_k=[H_{\lambda,n}\hat{\xi}_0^n]_k+\hat{G}_a,$$
where   $\hat{G}_a$ is independent with all components in  $[H_{\lambda,n}\hat{\xi}_0^n]$. Therefore $\eta_k^n$ can be expressed in the form of  \eqref{IGD-n}. Accordingly, conditions 1 and 2 of Theorem \ref{main theorem} are satisfied by the same argument as \eqref{130} and \eqref{J decrease}.  For condition 3 in Theorem \ref{main theorem}, let $R\geq0$, we have
\begin{align*}
	\E_{\eta^n_{1:k-1}}\left[ L_{f_{k|y}}(R) \right]&= \int_{\R^{k-1}} L_{f_{k|y}}(R) f_{1:k-1}(y)dy = \int_{\R^{k-1}} \int_{|x|\geq R}x^2f_{k|y}(x)dx f_{1:k-1}(y)dy \\
	& = \int_{|x|\geq R}x^2f_k(x)dx =  L_{\eta^n_k}(R).
\end{align*}
Let $\tilde{l}(R) = \min\{\frac{\sqrt{(\E \xi^4 + 3(\E \xi^2)^2)(\E \xi^2)}}{R},\E \xi^2\}$, then we can prove (see Appendix \ref{G} for the proof),
$$L_{\eta^n_k}(R) \leq \tilde{l}(R),\ \forall n\in\mathbb{N}^+,\ R\geq0,$$
verifying condition 3 .

Thus \eqref{polarhn} and \eqref{polarVn}   are obtained by Theorem \ref{main theorem} since $h_n(\omega)$ is convergent when $\omega\in O$. \eqref{polarhn1} is a direct consequence of  \eqref{polarhn} (cf. the proof of Corollary \ref{sum corollary} ).
\e
\subsection{ Limit of $V_n$ and $J_n$ }
To avoid unnecessary tedious discussions,  we give the definition below.
\begin{definition}
	If random variable $\xi$ has a density $p$ such that $p\in C_b^2(\R)$ and $\E \xi^4<\infty$, we say $\xi\in\mathcal{S},$ or equivalently, $p\in\mathcal{S}$.
\end{definition}
In what follow we shall confine our discussion  on  the random variables in $\mathcal{S}$.

\begin{lemma}\label{J down}
	Let $\xi, \eta$ be two independent random variables with densities,  then
	$$J(\xi*_\lambda \eta)\le \lambda^2 J(\xi)+(1-\lambda^2)J(\eta),$$
	$$J(\xi*_\lambda (-\eta))\le \lambda^2 J(\xi)+(1-\lambda^2)J(\eta),$$
	$$J(\xi*_\lambda (-\eta)|\xi*_\lambda \eta)=\lambda^2 J(\xi)+(1-\lambda^2)J(\eta),$$
\end{lemma}
\pf
The first two inequalities can be obtained directly from Lemma \ref{Fisher ineq},  we  prove only the third one.
Take
\begin{equation}f(x,y)= p_{\xi}^{1/2}(\dfrac{x+y}{\lambda})p_{\eta}^{1/2}(\dfrac{y-x}{\sqrt{1-\lambda^2}}),
\end{equation}
then
\begin{equation*}
\begin{aligned}
J(\xi*_\lambda (-\eta)|\xi*_\lambda \eta)&= 4\int \int \left( \dfrac{\partial}{\partial x}f(x,y)  \right)^2 dxdy
=\lambda^2 J(\xi)+(1-\lambda^2)J(\eta).
\end{aligned}
\end{equation*}\e

For the conditional variance, we have the following lemma.
\begin{lemma}\label{V down}
	Let $\xi, \eta, Z$ be three 
	random variables with densities, then
	$$\V(\xi|\eta)\ge \V(\xi|\eta,Z),$$
\end{lemma}
\pf It is sufficient to prove
\begin{equation*}\E_\eta[  \E^2(\xi|\eta)]\le \E_{(\eta,Z)}[  \E^2(\xi|\eta,Z)].\end{equation*}
In fact,
\begin{equation*}
\begin{aligned}
\E_{(\eta,Z)}[  \E^2(\xi|\eta,Z)]&= \int_\eta \int_Z \E^2(\xi|\eta=y,Z=z) p_{(\eta,Z)}(y,z) dzdy\\
&\ge \int_\eta   [   \int_Z \E(\xi|\eta=y, Z=z)   p_{Z|\eta}(z|y) dz  ]^2 p_\eta(y)dy\\
&= \int_\eta [ \E(\xi|\eta=y) ]^2p_\eta(y)dy\\
&=\E_\eta[  \E^2(\xi|\eta)].
\end{aligned}
\end{equation*}
where $\ge$ follows from Jensen's inequality, the second $=$ follows from Fubini theorem.\e

Thus the variance sequence $V_n$ and Fisher information sequence $J_n$ satisfy the following Lemmas.

\begin{lemma}\label{vnjn}
	Suppose that the input $\xi$ has density $p$, $p\in \mathcal{S}$ and $J(p)<\infty$, then $\forall \omega \in \Omega$, $J_n(\omega), V_n(\omega)$ are decreasing, and $V_nJ_n\ge1.$
\end{lemma}
\pf  First, if $\omega_n=0$, it is obvious that $V_n(\omega)=V_{n-1}(\omega)$.
According to  Lemma \ref{J down} we have
$$J_n(\omega)\le J_{n-1}(\omega).$$
If $\omega_n=1$, then by Lemma \ref{J down} and Lemma \ref{V down},
$$V_n(\omega)\le V_{n-1}(\omega),\;\; J_n(\omega)=J_{n-1}(\omega).$$
Thus we obtain the monotonically decreasing property of $J_n, V_n$.

Secondly, according to  \eqref{CR bound} and Cauchy-Schwartz inequality,
\begin{equation*}
\begin{aligned}
V_nJ_n&=\left(\int_{\R^{k-1}} f_{1:k(n,\omega)-1}(y)h\left(f_{k(n,\omega)|y}\right)dy \right)\left( \int_{\R^{k-1}} f_{1:k(n,\omega)-1}(y)\V\left(f_{k(n,\omega)|y}\right)dy\right)\\
&\ge \left( \int_{\R^{k-1}} f_{1:k(n,\omega)-1}(y)\left[h\left(f_{k(n,\omega)|y}\right)\V\left(f_{k(n,\omega)|y}\right)\right]^{1/2}dy\right)^{2}\\
&\ge \left( \int_{\R^{k-1}} f_{1:k(n,\omega)-1}(y) dy\right)^{2}=1.
\end{aligned}
\end{equation*}
\e
%

\begin{lemma}\label{vj>1}
	Suppose that the input $\xi$ has density $p$, $p\in \mathcal{S}$ and $J(p)<\infty$, then $\forall \omega \in \Omega$, $J(p)\ge J_n\ge\dfrac{1}{\V(p)}, \V(p)\ge V_n\ge\dfrac{1}{J(p)}.$
\end{lemma}
\pf It can be directly deduced by Lemma \ref{vnjn}.
\e

Finally, we come to the main conclusion concerning the limit of $V_n$ and $J_n$ as the theorem below.
\begin{theorem}\label{polar theorem1}
	If the input $\xi$ has form \eqref{IGD1} with $\E\hat{\xi}_0=0$ and $ \E\hat{\xi}^4_0<\infty$, then for all $\omega \in  \Omega$, we have
	
	\begin{equation}\label{polarJn}
		J_\infty(\omega)V_\infty(\omega)=1,
	\end{equation}
	and
	\begin{equation}\label{polarVbound}
		1/J(\xi)\le V_\infty(\omega)\le \V(\xi),
	\end{equation}	
	where $ V_\infty(\omega)=\lim_{n\to\infty} V_n(\omega),~~ J_\infty(\omega)=\lim_{n\to\infty} J_n(\omega)$.
\end{theorem}
\pf

 \eqref{polarJn} is a directly corollary of Theorem \ref{Fisher limit}, and
\eqref{polarVbound} comes from Lemma \ref{vj>1}.
\e


Below we make a rough estimate of the limit variance, which has close link to the error analysis in compression.
\begin{theorem}\label{V est}
	Under the conditions in Theorem \ref{polar theorem},
given any $\varepsilon$	satisfying $ \dfrac{1}{J(\xi)}<\varepsilon<\V(\xi)$, we have
	\begin{equation*}1-\dfrac{h(\xi)-1/2\log (2\pi e/J(\xi))}{ 1/2\log (\varepsilon J(\xi))}\le \P_\Omega\left(V_\infty(\omega)\le \varepsilon \right)\le \dfrac{1/2\log [2\pi e\V(\xi)]-h(\xi)}{1/2\log (\V(\xi)/\varepsilon)}.\end{equation*}
\end{theorem}
\pf Since $\E|h_n|\le C[\log(2\pi e (\V(\xi)+1))+\log(2\pi e a)]$, $h_n\xrightarrow{L^1}h_\infty, \forall \omega\in O$. Then using
$\E h_\infty(\omega)=h(\xi)$ , Lemma \ref{vj>1} and Chebyshev inequality.
\e

The theorem tells us that it is possible to  reduce the variances of the signals by Hadamard transform to a  very small value for a large proportion of polar paths, even if the input signal has large variance.  See an example in Appendix \ref{a ex}. These bounds can also give a theoretical estimate for the performance of Hadamard compression in \cite{ULE} when code length $N=2^n$ is sufficiently large.

\subsection{Theoretical study of Hadamard compression}

In Chapter 4 of \cite{ULE}, the author studied Hadamard compression for continuous input $\xi$ with form \eqref{IGD1}.

Let  $\bm{\eta}:=(\eta_1,\eta_2,\cdots, \eta_{2^n})^T= H_n\bm{\xi}:=H_n(\xi_1,\xi_2,\cdots,\xi_{2^n})^T$ be the received signals of $n$-iteration Hadamard transform. From numerical examples the author in \cite{ULE}
deduced that a high entropy subset of $\eta_i$'s are sufficient to recover all $\bm{\eta}$. This deduction was
verified by further examples of sampling and obtained a satisfactory performance. However, there was no theoretical deduction supporting this performance. For example, a natural question is why those  subset of $\eta_i$'s with lower entropy can be compressed.
Note that for continuous signals, low entropy does not necessarily imply small variance, and hence does not necessarily mean small uncertainty (see the example in Appendix \ref{a ex}).   Indeed, in the literature (see e.g. \cite{ULE}, \cite{2011Hadamardpaper}, \cite{ArikanrealF2021} and references therein), there was no theoretic study concerning Hadamard compression.

 Based on the research in this article, we can now conduct some theoretical analysis on Hadamard compression. Our analysis will consist of three recursive statements. The statements below are valid for almost every polarization path  $\omega \in \Omega$.

 {\it Statement 1.}
When $n$ tends to $\infty,$ the distribution of received signal $\eta^n_k(\omega)$  conditional on the values of  previous received signals  converges to a Gaussian distribution with corresponding conditional variance. This statement follows from \eqref{polarhn} or \eqref{polarhn1}.

 {\it Statement 2.}  The corresponding conditional variance $\V(\eta^n_k(\omega)|\eta^n_{1:k(\omega)-1})$ converges to a deterministic value $V_\infty(\omega).$ This statement follows directly from \eqref{polarVn}.

 {\it Statement 3.}  The conditional distribution  $(\eta^n_k(\omega)|\eta^n_{1:k(\omega)-1})$ is insensitive to the values of previous received signals in the sense that, for almost every realization $\eta^n_1=y_1,\eta^n_2=y_2,\cdots,\eta^n_{k-1}=y_{k-1},$ the conditional distribution of $(\eta^n_k(\omega)|\eta^n_1=y_1,\eta^n_2=y_2,\cdots,\eta^n_{k-1}=y_{k-1})$ converges  to  the  Gaussian distribution with the same variance  $V_\infty(\omega).$ This statement can be seen from the deduction below. For almost every $\omega\in \Omega$, when $n$ is sufficiently large, there exists $\sigma^n_k(\omega)>0$, where $k=k(n,\omega)$, such that for almost any realization $(y_1,y_2,\cdots, y_{k-1})$,
 \begin{align*}
	&\E h(\eta^n_k(\omega)|\eta^n_1,\eta^n_2,\cdots,\eta^n_{k-1})\overset{(a)}{\approx }\frac{1}{2}\log(2\pi e (\sigma^n_k(\omega))^2)\\
 	&\overset{(b)}{\approx }\frac{1}{2}\log(2\pi e \V(\eta^n_k|\eta^n_1=y_1,\eta^n_2=y_2,\cdots,\eta^n_{k-1}=y_{k-1})).
 \end{align*}
where $(a)$ comes from  \eqref{polarhn}, and $(b)$ comes from  \eqref{polarVn}.

 With the above statements, we can now conduct some theoretical analysis about the simulation results of \cite{ULE}.  By Statement 1, when $n$ is large enough, we may roughly regard the received signals conditional on the  knowledge of previous signals are Gaussian distributed variables. It is notably known that the entropy $h(\eta)$ of a Gaussian variable $\eta$ is uniquely determined by its variance $\sigma$, that is, $h(\eta)= \frac{1}{2}\log(2\pi e \sigma^2).$ Therefore, lower conditional entropy implies smaller conditional variance. Statement 2 says that, when $n$ is large enough, the conditional variance is nearly deterministic, and hence nearly does not depend on the condition of previous signals. Further, by Statement 3, the conditional distribution $\eta^n_k(\omega)$ will always converge  to  the Gaussian distribution with variance  $V_\infty(\omega),$ and will not be influenced by the specific values of previous signals.

 We now provide a solid theoretical analysis explaining the simulation results of \cite{ULE}. When $n$ is large enough,  for any $k,~1\le k \le 2^n,$ the received signal $\eta^n_k$ is roughly Gaussian distributed with variance, say $\sigma^n_k,$ and $\sigma^n_k$ is irrelevant to the previous received signals. It implies that, those $\eta^n_k$'s with lower entropy have less variance, and hence have smaller uncertainty. Consequently we can safely use  high entropy subset of $\eta^n_k$'s  to recover all the received signals.   Further more, since the conditional distribution is insensitive to the realization of previous signals,  when $n$ is large enough, the choice of compressed positions can be fixed for whatever input signals. This means  that the experimental Hadamard  compression performed  in \cite{ULE} exhibits universal applicability across diverse input signals.



\begin{appendix}

\end{appendix}


\begin{appendix}
\section{Some Properties of Hadamard Transform}\label{n=2}

\subsection{Proof of \eqref{basic1}}\label{A.1}
	
Below we use the notations introduced in Section \ref{section1.3}. We show first a simple example for Hadamard transform when $n=2$.
\begin{equation}\label{Hadamard n=2}
	\begin{aligned}
		&(\eta^2_1,\eta^2_2,\eta^2_3,\eta^2_4)^T=H_{\lambda ,2}(\xi_1,\xi_2,\xi_3,\xi_4)^T\\
		&=( (\xi_1\ast_\lambda \xi_2 )\ast_{\lambda}(\xi_3\ast_\lambda \xi_4), (\xi_1\ast_\lambda \xi_2 )\ast_{\sqrt{1-\lambda^2}}(-\xi_3\ast_\lambda \xi_4),\\
		&\;\;\;\;\;\;\;\;(\xi_1\ast_{\sqrt{1-\lambda^2}} (-\xi_2) )\ast_{\lambda}(\xi_3\ast_{\sqrt{1-\lambda^2}} (-\xi_4)),\\
		& (\xi_1\ast_{\sqrt{1-\lambda^2}} (-\xi_2) )\ast_{\sqrt{1-\lambda^2}}(\xi_3\ast_{\sqrt{1-\lambda^2}} (-\xi_4))) ^T\\
		&=(\eta^1_1\ast_\lambda \eta^1_3, \eta^1_1\ast_{\sqrt{1-\lambda^2}} (-\eta^1_3),\eta^1_2\ast_\lambda \eta^1_4, \eta^1_2\ast_{\sqrt{1-\lambda^2}} (-\eta^1_4) ),
	\end{aligned}
\end{equation}

We focus on four differential  entropies
\begin{equation}\label{four entropy}
	h(\eta^2_1),\;\;  h(\eta^2_2|\eta^2_1),\;\; h(\eta^2_3|\eta^2_{1:2}),\;\; h(\eta^2_4|\eta^2_{1:3}).
\end{equation}
Not that the last three are random variables.
By \eqref{Hadamard n=2}, we have
\begin{equation}
	\E h(\eta^2_2|\eta^2_1)=\E h(\eta^1_1\ast_{\sqrt{1-\lambda^2}}(-\eta^1_3)|\eta^1_1\ast_\lambda \eta^1_3),
\end{equation}
\begin{equation}
	\E h(\eta^2_3|\eta^2_{1:2})=\E h(\eta^1_2\ast_\lambda \eta^1_4|\eta^1_1\ast_\lambda \eta^1_3, \eta^1_1\ast_{\sqrt{1-\lambda^2}}( -\eta^1_3)),
\end{equation}
\begin{equation}
	\E h(\eta^2_4|\eta^2_{1:3})= \E h(\eta^1_2\ast_{\sqrt{1-\lambda^2}} (-\eta^1_4)|\eta^1_1\ast_\lambda \eta^1_3, \eta^1_1\ast_{\sqrt{1-\lambda^2}} (-\eta^1_3),\eta^1_2\ast_\lambda \eta^1_4).
\end{equation}
Since conditional entropy is invariant  under the orthogonal transformations and $H_\lambda$ is an orthogonal transformation, we can write
\begin{equation}
	\E h(\eta^2_3|\eta^2_{1:2})=\E h(\eta^1_2\ast_\lambda \eta^1_4|\eta^1_1,\eta^1_3),
\end{equation}
\begin{equation}
	\E h(\eta^2_4|\eta^2_{1:3})= \E h(\eta^1_2\ast_{\sqrt{1-\lambda^2}} (-\eta^1_4)|\eta^1_1,\eta^1_3,\eta^1_2\ast_\lambda \eta^1_4).
\end{equation}
Then the sums of conditional entropy hold that
\begin{equation}
	\begin{aligned}
		& h(\eta^2_1)+\E h(\eta^2_2|\eta^2_1)= h(\eta^1_1\ast_\lambda \eta^1_3)+\E h(\eta^1_1\ast_{\sqrt{1-\lambda^2}}(-\eta^1_3)|\eta^1_1\ast_\lambda \eta^1_3)\\
		&= h(\eta^1_1\ast_\lambda \eta^1_3, \eta^1_1\ast_{\sqrt{1-\lambda^2}}(-\eta^1_3))\\
		&\overset{(a)}{=}h(\eta^1,\eta^3)\\
		&\overset{(b)}{=}2h(\eta^1_1),
	\end{aligned}
\end{equation}
and
\begin{equation}
	\begin{aligned}
		&\E h(\eta^2_3|\eta^2_{1:2})+\E h(\eta^2_4|\eta^2_{1:3})=\E h(\eta^1_2\ast_\lambda \eta^1_4|\eta^1_1,\eta^1_3)\\
		&\;\;\;\;\;\;\;\;+\E h(\eta^1_2\ast_{\sqrt{1-\lambda^2}} (-\eta^1_4)|\eta^1_1,\eta^1_3,\eta^1_2\ast_\lambda \eta^1_4)\\
		&=\E h((\eta^1_2\ast_{\sqrt{1-\lambda^2}} (-\eta^1_4),\eta^1_2\ast_\lambda \eta^1_4)|\eta^1_1,\eta^1_3)\\
		&\overset{(c)}{=}\E h((\eta^1_2, \eta^1_4)|\eta^1_1,\eta^1_3)\\
		&\overset{(d)}{=}2\E h(\eta^1_2|\eta^1_1),
	\end{aligned}
\end{equation}
where $(a)$ and $(c)$ come from the invariance of entropy under the orthogonal transformations, $(b)$ and $(d)$ come from the fact that  $(\eta^1_1,\eta^1_2)$ and $(\eta^1_3,\eta^1_4)$ are \textit{i.i.d.}, according to the iteration relation of Kronnecker product.
Thus following the general entropy jump inequality (see \cite[Theorem 2]{2004Solution}  as an example),  we conclude that
\begin{equation}\label{core,n=2,1}
	|h(\eta^1_1)-\E h(\eta^2_2|\eta^2_1)|=| h(\eta^2_1)-h(\eta^1_1)|=h(\eta^1_1\ast_\lambda \eta^1_3)-h(\eta^1_1)=h(\eta^1_1\ast_\lambda (\eta^1_1)')-h(\eta^1_1),
\end{equation}
and
\begin{equation}\label{core,n=2,2}
	\begin{aligned}
		&|\E h(\eta^1_2|\eta^1_1)|-\E h(\eta^2_4|\eta^2_{1:3})|=|\E h(\eta^2_3|\eta^2_{1:2})-\E h(\eta^1_2|\eta^1_1)|\\
		&=\E h(\eta^1_2\ast_\lambda \eta^1_4|\eta^1_1,\eta^1_3)-\E h(\eta^1_2|\eta^1_1)=\E h(\eta^1_2\ast_\lambda (\eta^1_2)'|\eta^1_1,(\eta^1_1)')-\E h(\eta^1_2|\eta^1_1),
	\end{aligned}
\end{equation}
where $(\eta^1_i)'$ is an \textit{i.i.d.} copy of $\eta^1_i$.
Thus \eqref{basic1} is valid for $n=2.$

With the above simple example in hand, we now verify \eqref{basic1} for general $n$-iteration of  Hadamard Transform.
As described in Section \ref{section1.3}, let $\Omega:=\{0,1\}^\infty$ be the polarization path space. For
$\omega\in\Omega$, we define
\begin{equation}\label{decimal}
	k(n,\omega)=[\omega_1\omega_2\cdots\omega_n]_{2\to 10}+1,
\end{equation}
where $[\cdot]_{2\to10}$ indicates the conversion of binary to decimal. We denote
	\begin{equation}
		\eta^n_k(\omega)=\eta^n_{k(n,\omega)}.
	\end{equation}
Then for each $n$-iteration polarization path $(\omega_1,\omega_2,\cdots\omega_n),~ \eta^n_k(\omega)$ is exactly the $k$th component $\eta^n_k$ of $ \eta^n.$
By  \eqref{decimal}, if $\omega_{n+1}=0$ then $j=k(n+1,\omega)$ is odd .  Similar to the above  discussion of the case $n=2$, by the iteration relation of $H_{\lambda,n}$  and the  orthogonal property of $H_\lambda$, we have for $j=2l-1, l=1,2,\cdots, 2^n$,
\begin{equation*}
	\begin{aligned}
		&\E h[\eta^{n+1}_{j}|\eta^{n+1}_{1:{j-1}}]=\E  h[\eta^{n}_{l}\ast_\lambda \eta^n_{2^n+l}|\eta^{n+1}_{1:{j-1}}]\\
		&=\E h[\eta^{n}_{l}\ast_\lambda \eta^n_{2^n+l}|\eta^{n}_{1}\ast_\lambda \eta^n_{2^n+1},\eta^{n}_{1}\ast_{\sqrt{1-\lambda^2}} (-\eta^n_{2^n+1})\cdots,\\ &\;\;\;\;\;\;\;\; \eta^{n}_{l-1}\ast_\lambda \eta^n_{2^n+l-1},\eta^{n}_{l-1}\ast_{\sqrt{1-\lambda^2}} (-\eta^n_{2^n+l-1}) ]\\
		&=\E h[\eta^{n}_{l}\ast_\lambda \eta^n_{2^n+l}|\eta^{n}_{1},\eta^n_2\cdots, \eta^n_{l-1}, \eta^n_{2^n+1},\eta^n_{2^n+2},\cdots, \eta^n_{2^n+l-1} ].
	\end{aligned}
\end{equation*}

Similarly, if $\omega_{n+1}=1$ then $j=k(n+1,\omega)$ is even,  and thus for $j=2l$ we have
\begin{equation*}
	\begin{aligned}
		&\E h[\eta^{n+1}_{j}|\eta^{n+1}_{1:{j-1}}]\\
		&=\E h[\eta^{n}_{l}\ast_{\sqrt{1-\lambda^2}} (-\eta^n_{2^n+l})|\eta^{n+1}_{1:{j-1}}]\\
		&=\E h[\eta^{n}_{l}\ast_{\sqrt{1-\lambda^2}} (-\eta^n_{2^n+l})|\eta^{n}_{1}\ast_\lambda \eta^n_{2^n+1},\eta^{n}_{1}\ast_{\sqrt{1-\lambda^2}} (-\eta^n_{2^n+1}),\\
		&\;\;\;\;\;\; \cdots, \eta^{n}_{l-1}\ast_\lambda \eta^n_{2^n+l-1},\eta^{n}_{l-1}\ast_{\sqrt{1-\lambda^2}} (-\eta^n_{2^n+l-1}), \eta^{n}_{l}\ast_\lambda \eta^n_{2^n+l} ]\\
		&=\E h[\eta^{n}_{l}\ast_{\sqrt{1-\lambda^2}} (-\eta^n_{2^n+l})|\eta^{n}_{1},\eta^n_2\cdots, \eta^n_{l-1}, \eta^n_{2^n+1},\eta^n_{2^n+2},\cdots, \eta^n_{2^n+l-1},\eta^{n}_{l}\ast_\lambda \eta^n_{2^n+l}  ]\\
		&=\E h[\eta^{n}_{l}\ast_{\sqrt{1-\lambda^2}} (-\eta^n_{2^n+l})|\eta^{n}_{1},\eta^n_2\cdots, \eta^n_{l-1}, \eta^n_{2^n+1},\eta^n_{2^n+2},\cdots, \eta^n_{2^n+l-1},\eta^{n}_{l}\ast_\lambda \eta^n_{2^n+l}  ].
	\end{aligned}
\end{equation*}
Thus the sum of $j=2l-1$ and $j+1=2l$ satisfy that
\begin{equation}
	\begin{aligned}
		&\E h[\eta^{n+1}_{j}|\eta^{n+1}_{1:{j-1}}]+\E h[\eta^{n+1}_{j+1}|\eta^{n+1}_{1:{j}}]\\
		&=\E h[\eta^{n}_{l}\ast_\lambda \eta^n_{2^n+l}|\eta^{n}_{1},\eta^n_2\cdots, \eta^n_{l-1}, \eta^n_{2^n+1},\eta^n_{2^n+2},\cdots, \eta^n_{2^n+l-1} ]\\
		&\;\;\;\;+\E h[\eta^{n}_{l}\ast_{\sqrt{1-\lambda^2}} (-\eta^n_{2^n+l})|\eta^{n}_{1},\eta^n_2\cdots, \eta^n_{l-1}, \eta^n_{2^n+1},\eta^n_{2^n+2},\cdots, \eta^n_{2^n+l-1},\eta^{n}_{l}\ast_\lambda \eta^n_{2^n+l}  ]\\
		&=\E h[(\eta^{n}_{l}\ast_\lambda \eta^n_{2^n+l},\eta^{n}_{l}\ast_\lambda \eta^n_{2^n+l}) |\eta^{n}_{1},\eta^n_2\cdots, \eta^n_{l-1}, \eta^n_{2^n+1},\eta^n_{2^n+2},\cdots, \eta^n_{2^n+l-1} ]\\
		&=\E h[(\eta^{n}_{l}, \eta^n_{2^n+l}) |\eta^{n}_{1},\eta^n_2\cdots, \eta^n_{l-1}, \eta^n_{2^n+1},\eta^n_{2^n+2},\cdots, \eta^n_{2^n+l-1} ]\\
		&=2\E h(\eta^{n}_{l} |\eta^{n}_{1},\eta^n_2\cdots, \eta^n_{l-1}).
	\end{aligned}
\end{equation}
Eventually, similar to \eqref{core,n=2,1} and \eqref{core,n=2,2} in the case $n=2$, we can deduce that,
\begin{equation}\label{core for general n, APP}
	\begin{aligned}
		&|\E h(\eta^{n+1}_{k(n+1,\omega)}|\eta^{n+1}_{1:k(n+1,\omega)-1} )-\E h(\eta^{n}_{k(n,\omega)}|\eta^{n}_{1:k(n,\omega)-1} )|\\
		&=\E h(\eta^{n}_{k(n,\omega)}\ast_\lambda (\eta^n_{k(n,\omega)})'|\eta^{n}_{1:k(n,\omega)-1},(\eta^{n}_{1:k(n,\omega)-1})' )-\E h(\eta^{n}_{k(n,\omega)}|\eta^{n}_{1:k(n,\omega)-1} ),\;\;\forall n, \omega,
	\end{aligned}
\end{equation}
where $X'$ denotes an \textit{i.i.d.} copy of  $X$.  Thus \eqref{basic1}  is verified  for general $n$-iteration of  Hadamard Transform.



\subsection{Proofs of \eqref{basic2} and \eqref{basic3}\label{A.2}}
We still start with $n=2$. If $(\xi_1,\xi_2,\xi_3,\xi_4)$ has the form $\xi_i=\tilde{\xi}_i+G_i$ as in \eqref{IGD1}
, then
\begin{equation*}
\eta^2=H_2\tilde{\xi}+\bm{G},
\end{equation*}
and it is clearly that for any $k=1,2,3,4$, $\eta^2_k$ has the form $\eta^2_k=\hat{\eta}^2_k+G_k$, where
\begin{equation}
\hat{\eta}^2_k\overset{d}{=}\frac{\tilde{\xi}_1+\tilde{\xi}_2+\tilde{\xi}_3+\tilde{\xi}_4}{2}+G_1\mbox{~~or~~}\hat{\eta}^2_k\overset{d}{=}\frac{\tilde{\xi}_1+\tilde{\xi}_2-\tilde{\xi}_3-\tilde{\xi}_4}{2}+G_1,
\end{equation}
$G_k$ is Gaussian and independent with  $\hat{\eta}^2_k$, and $\eta^2_{1:k-1}$ (if $k>1$). Thus  by using (4.12)(a) in \cite{ULE} (note that our index starts from 1 and \cite{ULE}(4.12)(a)  from 0), we obtain the bounds of the four (random) entropies that for all $k=1,2,3,4,$
\begin{equation}\label{bounds for 4 entropy}
\frac{1}{2}\log(2\pi e \V(G_1))\le  h(\eta^2_k|\eta^2_{1:k-1})\le h(\eta^2_k) \le \frac{1}{2}\log(2\pi e (\V(\tilde{\xi}_1)+\V(G_1)).
\end{equation}

Now we turn to general $n$. We assume also that $\xi = \hat{\xi}_0 + G_a$ as  in \eqref{IGD1} .  Similar to \eqref{bounds for 4 entropy}, we can get the bounds of random variables $ h(\eta^{n}_{k(n,\omega)}|\eta^{n}_{1:k(n,\omega)-1})$ as below.
\begin{equation}\label{bounds to uni-intg}
	\frac{1}{2}\log(2\pi e a)\le  h(\eta^{n}_{k(n,\omega)}|\eta^{n}_{1:k(n,\omega)-1})\le \frac{1}{2}\log(2\pi e (\V(\hat{\xi}_0)+a)),\;\;\forall \omega, n.
\end{equation}
Next, as described at the beginning of Section \ref{section6},  we impose a probability measure $\P_\Omega$ on $(\Omega, \mathcal{F} )$ in such a way that  $\P_\Omega(S(b_1,b_2,\cdots,b_n))=1/2^n$  for each cylindrical set. For  $n\ge 1,$  let $\mathcal{F}_n$ be the $\sigma$-algebra generated by cylindrical sets $S(b_1,b_2,\cdots, b_i),1\le i\le n, b_1,\cdots,b_i\in\{0,1\}$, and let $\mathcal{F}_0$ be the trivial $\sigma$-algebra consisting of  null sets and $\Omega$ only. Then $\{\mathcal{F}_n, n\ge 0\}$ is a filtration.  For $n\ge 1,$ we denote
\begin{equation} h_n(\omega):=\E_{\eta^n_{1:k(n,\omega)-1}}  h(\eta^n_{k(n,\omega)}| \eta^n_{1:k(n,\omega)-1}  ).
\end{equation}
Then the random variable  $h_n(\omega)$ on the probability space $(\Omega, \mathcal{F},\P_\Omega)$ is $\mathcal{F}_n$- measurable. Similar to the discussion in \cite{Polar2009} and \cite{ULE}, we can show that $\{h_n\}_{n \ge 1}$ is a martingale about the filtration $\{\mathcal{F}_n, n\ge 0\}$.
In fact  $\{h_n\}_{n \ge 1}$ is a bounded martingale because by \eqref{bounds to uni-intg} we have
\begin{equation}\label{basic2a}
	\frac{1}{2}\log(2\pi e a)\leq	h_n\leq \frac{1}{2}\log(2\pi e \V(\xi)).
\end{equation}
Thus by the well known martingale convergence theorem (cf. e.g. Theorem 9.4.5 in \cite{ZKL}), we conclude that there exists an integrable random variable $h_\infty$ on $(\Omega,\mathcal{F},\P_\Omega)$ such that
\begin{equation}\label{martingale}
 \lim_{n\to\infty}h_n  = h_\infty, ~~a.s.~~ \omega \in \Omega .
\end{equation}

Going back to the notations we introduced in \eqref{Xn, Yn}):
\begin{equation*}\label{Xn, Yn, APP}
X_n(\omega)=\eta^n_{k(n,\omega)},\;\; Y_n(\omega)=(\eta^n_{1},\eta^n_{2},\cdots,\eta^n_{k(n,\omega)-1})=\eta^n_{1:k(n,\omega)-1},
\end{equation*}
we have proved \eqref{basic2} by \eqref{basic2a}, and proved \eqref{basic3} by \eqref{martingale}.

\section{Proof of Proposition \ref{bound on H and B}}\label{proof of bound on H and B}
	\begin{proof}   In the following, without loss of generality, we always assume $q = p_t^*f$ and $X$ be the random variable with density $f$. (Or just take an almost surely convergent subsequence and the desired results follow from dominated convergence theorem.)
	
			From \eqref{OU-prop2} we have
			\begin{equation}
			q(x) = \E[g_{1-e^{-2t}}(x-e^{-t}X)]\leq  \frac{1}{\sqrt{2\pi(1-e^{-2t})}}.
			\end{equation}
	
			2. Using \eqref{OU-prop2} we obtain
			\begin{equation}
			q'(x) = -\frac{1}{1-e^{-2t}}\E\left[(x-e^{-t}X)g_{1-e^{-2t}}(x-e^{-t}X)\right].
			\end{equation}
			By Cauchy-Schwartz inequality,
			\begin{equation}
			[q'(x)]^2 \leq \frac{1}{(1-e^{-2t})^2}q(x)\E\left[(x-e^{-t}X)^2g_{1-e^{-2t}}(x-e^{-t}X)\right].
			\end{equation}
			Let
			\begin{equation}
			U_t(x) = \frac{x^2g_{1-e^{-2t}}(x)}{g_{1-e^{-4t}}(x)}.
			\end{equation}
			Clearly $\sup\limits_{y\in \R} U_t(y) < \infty$.  Denote $B_t = \frac{\sup\limits_{y\in \R} U_t(y)}{(1-e^{-2t})^2}$,  then
			\begin{equation}
			[q'(x)]^2 \leq B_tq(x)\E\left[g_{1-e^{-4t}}(x-e^{-t}X)\right] = B_tq(x)[p_{2t}^*(e^tX)](x).
			\end{equation}
		
			3. For $s\in(0,t)$, let $\varphi_{s, t}(x, y) = \frac{g_t(x+y)}{g_s(x)}$, and $\alpha(s,t,M) = \inf\limits_{x\in\R,|y|\leq M} \varphi_{s,t}(x,y)$. It is not hard to see that $\alpha(s,t,M) > 0$ for any $M \geq 0$. Therefore, take $s = 1 - e^{-t}$ and $M > 1$ such that $l(M) < 1$, we have
			\begin{equation}
			\begin{aligned}
			q(x) &= \E\left[g_{1-e^{-2t}}(x-e^{-t}X)\right] \geq \E\left[g_{1-e^{-2t}}(x-e^{-t}X)\mathbf{1}_{\{|X|\leq M\}}\right]\\
			&\geq\P(|X|\leq M)\alpha(1-e^{-t}, 1-e^{-2t}, e^{-t}M)g_{1-e^{-t}}(x)\\
			&\geq \left(1 - \E[X^2\mathbf{1}_{\{|X|\geq M\}}]\right)\alpha(1-e^{-t}, 1-e^{-2t}, e^{-t}M)g_{1-e^{-t}}(x)\\
			& = (1 - l(M))\alpha(1-e^{-t}, 1-e^{-2t}, e^{-t}M)g_{1-e^{-t}}(x).
			\end{aligned}
			\end{equation}
			The desired result follows by taking $A_{t, l} = (1-l(M))\alpha(1-e^{-t}, 1-e^{-2t}, e^{-t}M)$.
	\end{proof}
	
	\section{Proof of Proposition \ref{L of HtL}}\label{proof of L of HtL}
	\begin{proof}   For any $t \geq 0, q\in\mathcal{H}_{t, l}$, let $q = p_t^*f$ and $X$ be the random variable with density $f$. Let $L_G$ be the tail variance of standard Gaussian, then for any $ R > 2$,
		\begin{equation}
		\begin{aligned}
		&\int_{|x|\geq R} x^2q(x)dx = \E\left[(e^{-t}X+\sqrt{1-e^{-2t}}G)^2\mathbf{1}_{\{|e^{-t}X+\sqrt{1-e^{-2t}}G|\geq R\}}\right]\\
		&\leq 2\E\left[(e^{-2t}X^2+(1-e^{-2t})G^2)(\mathbf{1}_{\{|e^{-t}X|\geq R/2\}}+\mathbf{1}_{\{|\sqrt{1-e^{-2t}}G|\geq R/2\}})\right]\\
		&\leq 2\left[ e^{-2t}l\left( R/2\right) + e^{-2t}l(0)\P\left(|G|\geq  R/2\right)
		+ (1-e^{-2t})\P\left(|X|\geq  R/2\right) + (1-e^{-2t})L_G\left( R/2\right)\right]\\
		& \leq 2\left[e^{-2t}l\left( R/2\right) + e^{-2t}l(0)L_G\left( R/2\right)
		+ (1-e^{-2t})l\left( R/2\right) + (1-e^{-2t})L_G\left( R/2\right)\right]\\
		& \leq 2l\left( R/2\right) + 2(l(0)+1)L_G\left( R/2\right).
		\end{aligned}
		\end{equation}
		For $R\in[0,2]$,
		$$\int_{|x|\geq R} x^2q(x) \leq \int_{|x|\geq 0} x^2q(x) = \V(p_t^*f) = e^{-2t}\V(f) + 1 - e^{-2t} \leq l(0) + 1.$$
		Take
		\begin{equation}
		l'(R) = \left\{\begin{aligned} & \max\{l(0) + 1\ ,\ 2l(1)+2(l(0)+1)L_G(1)\}\ ,\ 0\leq R \leq 2 \\
		& 2l\left( R/2\right) + 2(l(0)+1)L_G\left( R/2\right)\ ,\ R > 2\end{aligned} \right. ,
		\end{equation}
		then $l'$ decreases to 0 and  $L_q\leq l'$. \e
		
		\section{Proof of Proposition \ref{continuity on h V KL}}\label{proof of continuity on h V KL}
		\pf  1. By Proposition \ref{bound on H and B}, there exists a constant $A > 0$ such that
		\begin{equation}
		e^{-A(1+x^2)} \leq q(x) \leq e^{A(1+x^2)}, \ \forall q\in\mathcal{H}_{t, l}.
		\end{equation}
		Thus
		\begin{equation}
		|\log q(x)| \leq A(1+x^2),\ \forall q\in\mathcal{H}_{t, l}.
		\end{equation}
		Suppose $q_n, q\in\mathcal{H}_{t, l}$ and  $\|q_n-q\|_{L^1_{1+x^2}}\rightarrow 0$, then
		\begin{equation}
		\begin{aligned}
		&|h(q_n) - h(q)| \leq \int_\R |q(x)\log q(x) - q_n(x)\log q_n(x)|dx\\
		&\leq \int_\R q(x)|\log q(x) - \log q_n(x)|dx + \int_\R |q(x) - q_n(x)||\log q_n(x)|dx \\
		&\leq \int_\R q(x)|\log q(x) - \log q_n(x)|dx + A\|q_n-q\|_{L^1_{1+x^2}}  \\
		&:= I_n + A\|q_n-q\|_{L^1_{1+x^2}}.
		\end{aligned}
		\end{equation}
		For any subsequence $n_k$, there exists a further subsequence $n_{k_j}$ such that $q_{n_{k_j}}\xrightarrow{a. e. } q$ as $q_{n_k} \xrightarrow{L^1} q$. Since
		\begin{equation}
		q(x)|\log q(x) - \log q_{n_{k_j}}(x)| \leq 2Aq(x)(1+x^2) \in L^1,
		\end{equation}
		we conclude that $I_{n_{k_j}}\rightarrow 0$ by dominated convergence theorem. The above argument shows that for any subsequence $I_{n_k}$,  we can find a further subsequence $I_{n_{k_j}}$ converging to 0,  which implies $I_n\rightarrow 0$, and hence $|h(q_n) - h(q)|\rightarrow 0$.
		
		2. For any $f,g\in\mathcal{H}_{t, l}$ we have
		\begin{equation}
		|\V(f) - \V(g)| \leq \int_\R x^2|f(x)-g(x)|dx \leq \|f-g\|_{L^1_{1+x^2}}
		\end{equation}
		
		3. Because $D(f) = \frac{1}{2}\log(2\pi e\V(f)) - h(f)$, the desired result follows from the first two parts of Proposition \ref{continuity on h V KL}.
		\section{Proof of Proposition \ref{F and L condition}}\label{proof of F and L condition}
		Fix $ \epsilon >0$, take $J_1 = \frac{2J_0}{\epsilon}$. Recall $J(X_n|Y_n)\leq J_0$, we have
		\begin{equation}\label{J}
		\P\left(J(X_n|Y_n) \leq J_1\right)\geq 1 -  \frac{\E J(X_n|Y_n)}{J_1} \geq  1 - \frac{\epsilon}{2}.
		\end{equation}
		Since $l(R)$ decreases to 0 as $R\rightarrow \infty$, it is able to take a strictly increasing sequence $\{R_m,m\geq 0\}$ such that $R_0 = 0$ and
		\begin{equation}
		l(R_m) \leq \min\left\{\frac{1}{m2^m},\frac{l(R_{m-1})}{2}\right\},\ \ \forall m\geq 1.
		\end{equation}
		Define
		\begin{equation}
		l_1(R) = \frac{2^{m+1}}{\epsilon}l(R_m),\ \ \text{if}\ R\in[R_m,R_{m+1}),\ m=0,1,\cdots.
		\end{equation}
		Clearly we have
		\begin{align}
		&l_1(R_{m}) \leq \frac{2^{m+1}}{\epsilon}\frac{l(R_{m-1})}{2} = \frac{2^{m}}{\epsilon}l(R_{m-1}) = l_1(R_{m-1}),\\
		&l_1(R_{m}) \leq  \frac{2^{m+1}}{\epsilon}\frac{1}{m2^m} \xrightarrow{m\rightarrow\infty} 0.
		\end{align}
		Therefore, $l_1(R)$ decreases to 0 as $R\rightarrow\infty$. For each $R_m$ we have
		\begin{equation}
		\P(L_{X_n|Y_n}(R_m)> l_1(R_m)) \leq \frac{\E\left[ L_{X_n|Y_n}(R_m) \right]}{l_1(R_m)} = \frac{l(R_m)}{l_1(R_m)} = \frac{\epsilon}{2^{m+1}}.
		\end{equation}
		Consequently,
		\begin{equation}\label{L}
		\begin{aligned}
		&\P\left(\exists R\geq 0\ s.t.\ L_{X_n|Y_n}(R)> l(R)\right) \leq \P\left(\bigcup_{m=0}^\infty\{L_{X_n|Y_n}(R_m)> l(R_m)\} \right)\\
		&\leq \sum\limits_{m=1}^\infty \frac{\epsilon}{2^{m+1}} = \frac{\epsilon}{2}.
		\end{aligned}
		\end{equation}
		The desired result follows from \eqref{J} and \eqref{L}.
	\end{proof}
	
	\section{Uniformly integrability of $\{\bar{W}_n,n\ge 1\}$}\label{proof of 2th moment}
\begin{lemma}\label{weak to strong converge}
	Suppose that  $\xi_n$ are \textit{i.i.d.} samples from $\xi$ and $\E \xi=0, \E \xi^2=\sigma^2<\infty$, then the sequence $\{ W^2_n, n\ge 1\}$ is uniformly integrable, where
	$$W_n:=\frac{1}{\sqrt{n}}\sum_{i=1}^n \xi_i.$$
\end{lemma}
\pf By CLT  and Skorohod representation theorem, there exists a Gaussian random variable $G_\sigma\sim\mathcal{N}(0,\sigma^2)$ such that $W_n\overset{a.s.}{\rightarrow}G_\sigma$, and thus
\begin{equation}\label{Pr converg}
  W^2_n\overset{Pr.}{\longrightarrow}G^2_\sigma.
\end{equation}
Since we also have that
$$\E W_n^2=\sigma^2=\E G_\sigma^2,\;\;\forall n\ge 1,$$
we obtain the uniformly integrability of  $\{ W^2_n, n\ge 1\}$.
 \e

\begin{proposition}
	Let $(\xi,\eta)$ be a two-dimensional random variable such that $\E \xi^2 < \infty$, $\bar{W}_n = \frac{\sum\limits_{k=1}^n (\xi_k - \E(\xi_k|\eta_k))}{\sqrt{n}}$, where $(\xi_k,\eta_k)$ are \textit{i.i.d.} samples from $(\xi,\eta)$. Then the sequence $\left\{\bar{W}_n^2,n\geq1\right\}$ is uniformly integrable.
\end{proposition}
\pf It is easy to see that $\bar{\xi}_k:=\xi_k - \E(\xi_k|\eta_k)$ are \textit{i.i.d.} samples from $\bar{\xi}:=\xi-\E(\xi|\eta)$,  and
$$\E\bar{\xi}=0,\; \E\bar{\xi}^2\le 4\E \xi^2<\infty$$
due to Cauchy-Schwartz inequality and Jensen's inequality. By Lemma \ref{weak to strong converge}, the proof is completed.  \e
	
	\section{Moment and Tail Variance Bounds of Independent Sum.}\label{G}
	\begin{proposition}\label{moment and L bound}
		Suppose that $X_1,X_2,...X_n\overset{\textit{i.i.d.}}{\sim} X$, $X$  has density $f(x)\in\mathcal{S}$ and $\E X = 0$. Let $0<\lambda_1,\lambda_2,...,\lambda_n<1$ be such that $\sum\limits_{i=1}^n \lambda_i^2 = 1$. Set $Y = \sum\limits_{i=1}^n \lambda_iX_i$, then
		\begin{enumerate}
			\item $\E Y^2 = \E X^2,\  \E Y^4 \leq \E X^4 + 3(\E X^2)^2$.
			\item $L_Y(R) \leq \min\{\frac{\sqrt{(\E X^4 + 3(\E X^2)^2)(\E X^2)}}{R},\E X^2\}$.
		\end{enumerate}
	\end{proposition}
	\begin{proof}  1. $\E Y^2 = \E\left[\sum\limits_{i=1}^n \lambda_i X_i\right]^2 = \E X^2\sum\limits_{i=1}^n \lambda_i^2 = \E X^2$.
		\begin{align*}
		&\E Y^4  = \E\left[\sum\limits_{i=1}^{n} \lambda_iX_i\right]^4 =\E X^4\sum\limits_{i=1}^{n} \lambda_i^4 +3 (\E X^2)^2 \sum\limits_{i\neq j}\lambda_i^2\lambda_j^2 \\
		& = \E X^4\sum\limits_{i=1}^{n} \lambda_i^4 +3 (\E X^2)^2 \sum\limits_{i=1}^{n} \lambda_i^2(1-\lambda_i^2)\leq (\E X^4+3(\E X^2)^2)\sum\limits_{i=1}^{n} \lambda_i^2 \\
		&= \E X^4+3(\E X^2)^2.
		\end{align*}
		
		2.
		\begin{align*}
		L_Y(R) & = \E[Y^2\mathbf{1}_{\{|Y|\geq R\}}] \leq\sqrt{ \P(|Y|\geq R)\E Y^4} \leq \sqrt{(\E X^4 + 3(\E X^2)^2) \frac{\E Y^2}{R^2}}\\
		& = \frac{\sqrt{(\E X^4 + 3(\E X^2)^2)(\E X^2)}}{R}.
		\end{align*}
		Thus $L_Y(R)\leq \min\{L_Y(0),\frac{\sqrt{(\E X^4 + 3(\E X^2)^2)(\E X^2)}}{R}\} = \min\{\E X^2,\frac{\sqrt{(\E X^4 + 3(\E X^2)^2)(\E X^2)}}{R}\}$.
	\end{proof}

\section{An Example for Theorem \ref{V est}}\label{a ex}
Given a distribution $A$, let $h_0:=h(A)$ and $J_0:=J(A)$. We assume that the density function $p$ of $A$ is continuous and differentiable, and there exists $x_1< x_0<x_2$ \textit{s.t.}
\begin{equation*}
p(x_1)>0,\;\; p(x_0)=0,\;\; p'(x_0)=0,\;\; p(x_2)>0.
\end{equation*}
We define $A_i:=\frac{1}{i} A$, then it follows that $\V(A_i)\rightarrow 0$ as $i\to \infty$. We further define a random variable $B(i,j)$ using the function
\begin{equation*}
q_{i,j}(x):=\begin{cases}
p_i(x+j), & x\le x_0-j;\\
0, & x_0-j<x<x_0;\\
p_i(x), &x\ge x_0,
\end{cases}
\end{equation*}
as its density, where $p_i$ is the density function of $A_i$. It is not difficult to observe that for any $i$,
\begin{equation*}
\V(B(i,j))\rightarrow\infty \; \text{as} \; j\to \infty
\end{equation*}
and
\begin{equation*}
h(B(i,j))=h(A_i)=h_0-\log(i);\; J(B(i,j))=J(A_i)=i^2J_0.
\end{equation*}
Now, for any $\varepsilon, \varepsilon'>0$, we can find large $i$  such that
\begin{equation}\label{1}
\dfrac{h(B(i,j))-\frac{1}{2}\log [2\pi e(1/J(B(i,j)))]}{\frac{1}{2}\log [2\pi e\varepsilon]- \frac{1}{2}\log [2\pi e(1/J(B(i,j)))]}=\dfrac{h_0-\frac{1}{2}\log [2\pi e(1/J_0)]}{\frac{1}{2}\log (\varepsilon i^2J_0)}\le \frac{1}{2}\varepsilon'.
\end{equation}
Note that \eqref{1} holds uniformly  about $j$.
Finally, we note that for $B(\tau):=B(i,j)+a_\tau,\;\; a_\tau\sim N(0,\tau),$ the entropy, the Fisher information and the variance  are all continuous (differentiable) with respect to $\tau$, therefore, we can find a sufficiently small $\tau$ such that when using $B=B(\tau)$ as input, the Hadamard transformation effectively reduces the large variance to be almost smaller than $\varepsilon$, \textit{i.e.}
\begin{equation*}
\P_\Omega(V_\infty(\omega)\le \varepsilon)\ge 1-\varepsilon'.
\end{equation*}

  \end{appendix}
\begin{acks}[Acknowledgments]  The authors are very grateful to Qi-Man Shao and Yuan Li for their stimulating discussions. The  valuable suggestions of Qi-Man Shao greatly improves this research. Yuan Li helped us a lot for revising the earlier version of this paper.

\end{acks}
\begin{funding}
The research is supported by National Key R\&D Program of China No. 2023YFA1009603.

\end{funding}


\end{document}